%% file: spp_article.tex
\documentclass[article,onefignum,onetabnum]{siamart190516}

\input{spp_shared}

\ifpdf
\hypersetup{
  pdftitle={A semismooth Newton stochastic proximal point algorithm with variance reduction},
  pdfauthor={A. Milzarek, F. Schaipp and M. Ulbrich}
}
\fi

\begin{document}

\maketitle

\begin{abstract}
We develop an implementable stochastic proximal point (SPP) method for a class of weakly convex, composite optimization problems.  
The proposed stochastic proximal point algorithm incorporates a variance reduction mechanism and the resulting SPP updates are solved using an inexact semismooth Newton framework. We establish detailed convergence results that take the inexactness of the SPP steps into account and that are in accordance with existing convergence guarantees of (proximal) stochastic variance-reduced gradient methods. Numerical experiments show that the proposed algorithm competes favorably with other state-of-the-art methods and achieves higher robustness with respect to the step size selection.
\end{abstract}



\section{Introduction}
\label{sec:introduction}

In this work, we consider optimization problems of the form
\begin{equation}
\label{prob:deterministic}
\min_{x \in \R^n}~\psi(x):= f(x) + \phi(x), \quad f(x) := \oneover{N}{\sum}_{i=1}^{N}~f_i(A_ix),
\end{equation}
where $A_i \in \R^{m_i \times n}$ is given and the component functions $f_i : \R^{m_i} \to \R$, $m_i\in\N$, $i \in \{1,\dots,N\}$, are supposed to be continuously differentiable. The mapping $\phi : \Rn \to \Rex$ is convex, proper, and lower semicontinuous. 

Problems of the form \cref{prob:deterministic} have gained significant attention over the past decades in the context of large-scale machine learning \cite{Bottou2010}. Initiated by the pioneering stochastic approximation methods of Robbins and Monro \cite{Robbins1951} and Kiefer and Wolfowitz \cite{Kiefer1952}, more recent  extensions of stochastic (proximal) gradient descent schemes for \cref{prob:deterministic} have been studied, including variance reduction methods such as \texttt{SAGA} \cite{Defazio2014} or \texttt{SVRG} \cite{Johnson2013,Xiao2014}, adaptive methods \cite{Duchi2011,Kingma2015} and several variants \cite{Milzarek2019,Hanzely2018,ShalevShwartz2013}.

In this work, we consider the stochastic proximal point method (SPP) for problem \cref{prob:deterministic}, which can be implicitly expressed as 
\[x^{k+1} = \prox{\alpha_k\phi}(x^k - \alpha_k \nabla f_{\mathcal{S}_k}(x^{k+1})),\]
where $\alpha_k >0$ is a suitable step size and $\nabla f_{\mathcal{S}_k}$ is a stochastic approximation of $\nabla f$.
In this article, we combine the SPP method with a \texttt{SVRG}-type variance reduction strategy and derive novel convergence guarantees. A particular focus is put on the question of how to efficiently compute the SPP update which is an implicit equation in $x^{k+1}$.

\subsection{Background and Related Work} If the function $\psi$ is convex, the famous (deterministic) proximal point algorithm (PPA) for minimizing \cref{prob:deterministic} can be represented as follows
\begin{equation} \label{eq:ppa} x^{k+1} = \prox{\alpha_k\psi}(x^k) := \argmin_{y \in \Rn}~\psi(y) + \frac{1}{2\alpha_k} \|x^k-y\|^2_2, \quad\quad \alpha_k > 0. \end{equation}
Here, the mapping $\prox{\alpha_k\psi} : \Rn \to \Rn$ denotes the well-known proximity operator of the function $\alpha_k \psi$. 
Convergence of the method \cref{eq:ppa} was first studied by Martinet in \cite{Martinet1972} and later extended by Rockafellar in his seminal paper \cite{Rockafellar1976a}; let us also refer to G\"uler \cite{Gueler1991}. Specifically, Rockafellar showed that an inexact version of PPA \cref{eq:ppa}, 
\begin{align} \label{eq:inex-ppa} x^{k+1} \approx \prox{\alpha_k\psi}(x^k), \end{align}
converges whenever one of the following accuracy conditions \vspace{1ex}
\begin{itemize}
\item $\|x^{k+1} - \prox{\alpha_k\psi}(x^k) \| \leq \varepsilon_k$ with $\sum_{k=0}^\infty \varepsilon_k < \infty$ or \vspace{1ex}
\item $\|x^{k+1} - \prox{\alpha_k\psi}(x^k) \| \leq \delta_k \|x^{k+1} - x^k\|$ with $\sum_{k=0}^\infty \delta_k < \infty$, \vspace{1ex}
\end{itemize}
is utilized. Furthermore, under an error bound condition, it is also possible to establish linear convergence if the regularization parameters $\alpha_k$ are chosen sufficiently large, see, e.g., \cite[Thm.\ 2]{Rockafellar1976a}. These strong theoretical properties and the duality-type results derived in \cite{Rockafellar1976} are the foundations of several semismooth Newton-based augmented Lagrangian approaches and PPAs developed in \cite{Zhao2010,Jiang2013, Jiang2014,Yang2015,Chen2016,Li2018} for semidefinite programming, nuclear and spectral norm minimization, and Lasso-type problems. 

More recently, stochastic versions of the PPA have been studied, in particular when the objective is in the form of an expectation. In this line of research, a model-based stochastic proximal point method has been proposed by Asi and Duchi \cite{Asi2019, Asi2019a} for the convex case and by Davis and Drusvyatskiy for the weakly convex, composite case \cite{Davis2019}. SPP has also been analyzed in an incremental framework \cite{Bertsekas2011} and for constrained problems with Lipschitz or strongly convex objective \cite{Patrascu2017}. To the best of our knowledge an effective combination of SPP and variance reduction techniques seems to be unavailable so far.

\subsection{Contributions}
Our main contributions and the core challenges addressed in this article are as follows:
\begin{itemize}
\item We introduce \texttt{SNSPP}, a semismooth Newton stochastic proximal point method with variance reduction for the composite problem \cref{prob:deterministic}. Similar to \texttt{SVRG} \cite{Xiao2014,J.Reddi2016}, we prove linear convergence in the strongly convex case and a sublinear rate is established in the weakly convex case (using constant step sizes). 
\item Semismooth Newton-based PPAs are known to be highly efficient in deterministic problems \cite{Zhao2010,Yang2015,Li2018,Zhang2020}. Our proposed algorithmic strategy benefits from the fast local convergence properties of the semismooth Newton method and allows to further reduce the computational complexity of the occurring subproblems which can be directly controlled by the selected batch sizes. 
\item The inexactness of the stochastic proximal steps is an integral component of our theoretical investigation. We present a unified analysis allowing broader applicability of the (variance reduced) SPP method.
\item Numerical experiments suggest that SPP with variance reduction performs favorably and is more robust in comparison to state-of-the-art algorithms.  
\end{itemize}
\section{Preliminaries}
\label{sec:preliminaries}
For $N \in \N$, we set $[N] := \{1,\dots,N\}$ and denote by $I \in \R^{n \times n}$ the identity matrix. By $\iprod{\cdot}{\cdot}$ and $\|\cdot\|$, we denote the standard Euclidean inner product and norm. The set of symmetric, positive definite $n \times n$ matrices is denoted by $\Spp$. For a given matrix $M \in \Spp$, we define the inner product $\langle x,y \rangle_M := \langle x, My\rangle$ and the induced norm $\|x\|_M := \sqrt{\iprod{x}{x}_M}$. The function $f : \Rn \to \Rex$ is called $\rho$-weakly convex, $\rho > 0$, if the mapping $x \mapsto f(x) + \frac{\rho}{2}\|x\|^2$ is convex. Furthermore, $f$ is called $\mu$-strongly convex, $\mu > 0$, if $f - \frac{\mu}{2}\|\cdot\|^2$ is a convex function. The set $\dom (f) := \{x \in \Rn: f(x) < +\infty\}$ denotes the effective domain of $f$.
The mapping $f$ is called $L$-smooth, if $f$ is differentiable on $\Rn$ and if there exists $L\geq 0$ such that
\begin{align*}
\|\nabla f(y) -\nabla f(x)\| \leq L \|y-x\| \quad \forall~x,y \in \R^n.
\end{align*}
The proximity operator and the Moreau envelope, defined below, are essential ingredients used in this work.
\begin{definition} \label{def:prox}
Let $g: \R^n \to \Rex$ be a proper and closed function and let $M \in \Spp$ be given such that the mapping $x\mapsto g(x) + \onehalf\|x\|^2_M$ is strongly convex. The proximity operator of $g$ is defined via
\begin{align*}
\label{eqn:defn_prox}
\prox{g}^M : \Rn \to \Rn, \quad \prox{g}^M(x) := \argmin_{z\in\Rn}~g(z) + \onehalf \|x-z\|^2_M .
\end{align*}
In addition, suppose that the function $g+\tfrac{1}{2\lambda}\|\cdot\|_M^2$ is strongly
convex for some $\lambda > 0$ and $M \in \Spp$. The associated Moreau envelope of $g$ is then given by
\begin{align}
\label{eqn:defn_moreau}
\mathrm{env}_g^{M,\lambda} : \Rn \to \R, \quad \mathrm{env}_g^{M,\lambda}(x) := \min_{z\in \Rn}~g(z) + \oneover{2\lambda}\|x-z\|^2_M.
\end{align}
        If $g$ is convex, we also use the notations $\prox{g}(x) = \prox{g}^I(x)$ and $\mathrm{env}_g(x) = \mathrm{env}_g^{I,1}(x)$.
\end{definition}
The proximity operator can be uniquely characterized by the optimality conditions of its underlying minimization problem. Let $M$ be positive semidefinite; if $g$ is prox-bounded with proximal subdifferential $\partial g$ (cf.\ \cite[Def.\ 8.45, Prop.\ 8.46]{Rockafellar1998}), then
\begin{equation} \label{eq:prox-char} 
p = \prox{g}^{I+M}(x) \quad \iff \quad p \in x - M(p-x) - \partial g(p).
\end{equation}
%
If $g$ is convex, the proximity operator is firmly nonexpansive \cite[Thm.\ 6.42]{Beck2017}, i.e.,
\begin{equation} \label{eq:nonexp} \|\prox{g}(x) - \prox{g}(y)\|^2 \leq \iprod{x-y}{\prox{g}(x)-\prox{g}(y)} \quad \forall~x,y \in \Rn. \end{equation}
In particular, $\prox{g}$ is Lipschitz continuous with constant $1$. Moreover, by \cite[Thm.\ 6.60]{Beck2017} the Moreau envelope is continuously differentiable with
\begin{align}\label{eqn:moreau_gradient}
\nabla \mathrm{env}_g(x) = x - \prox{g}(x).
\end{align}
For extensive discussions of the proximity operator, the Moreau envelope, and related concepts, we refer to \cite{Moreau1965,Bauschke2011,Beck2017}. 
For a function $g: \R^n \to \Rex$, the conjugate of $g$ is defined by $g^\ast:\R^n\to \Rex$, $g^\ast(x) := \sup_{z\in \R^{n}}~\langle z,x\rangle -g(z)$.
\begin{proposition}
\label{prop:strongly_convex_lipschitz}
Let $g: \R^n \to \Rex$ be proper and closed. If $g$ is $\mu$-strongly convex, then its conjugate $g^\ast$ is closed, convex, proper, and Fr\'{e}chet differentiable and its gradient is given by $\nabla g^\ast(x) = \argmax_{z\in \R^{n}}\,\langle z,x\rangle -g(z)$.
In addition, $\nabla g^\ast : \Rn \to \Rn$ is Lipschitz continuous with Lipschitz constant $\mu^{-1}$.
\end{proposition}
\begin{proof}
The first part is a consequence of \cite[Cor.\ 4.21]{Beck2017} and \cite[Prop.\ 17.36]{Bauschke2011}. The remaining properties follow from \cite[Prop.\ 13.11, Thm.\ 13.32, and Thm.\ 18.15]{Bauschke2011}.
\end{proof}

Let $F:\R^n \to \R^m$ be a locally Lipschitz function. We use $\partial F$ to denote its Clarke subdifferential ($\partial F$ reduces to the standard subdifferential if $m=1$ and $F$ is convex). As computing elements of $\partial F$ can be challenging, we will make use of generalized derivatives $\hat{\partial} F$ that share similar properties. 
We call $\hat\partial F$ a \textit{surrogate generalized differential} of $F$. Following \cite{Qi1993,Ulbrich2011}, we present a definition of semismoothness of $F$. 

\begin{definition}
Let $F: V \to \R^m$ be locally Lipschitz and let $V\subset \R^n$ be an open set. $F$ is called semismooth at $x\in V$ (with respect to $\partial F$), if $F$ is directionally differentiable
at $x$ and if it holds that
\begin{displaymath}
{\sup}_{M \in \partial F(x+s)}~\|F(x+s) -F(x) - Ms\| = o(\|s\|) \quad \text{as}~ s\rightarrow 0.
\end{displaymath}
Moreover, for $\nu>0$, $F$ is called $\nu$-order semismooth
(strongly semismooth if $\nu=1$) at $x\in V $ (w.r.t. $\partial F$), if $F$ is directionally differentiable
at $x$ and we have
\begin{displaymath}
{\sup}_{M \in \partial F(x+s)}~\|F(x+s) -F(x) - Ms\| = \mathcal{O}(\|s\|^{1+\nu}) \quad \text{as}~ s\rightarrow 0.
\end{displaymath}
\end{definition}
For problem \cref{prob:deterministic}, we introduce the proximal gradient mapping as a measure of stationarity, i.e., for $\alpha>0$, we define
\[F^\alpha_\mathrm{nat}: \Rn \to \Rn, \quad F^\alpha_\mathrm{nat}(x) := x-\prox{\alpha \phi}(x-\alpha \nabla f(x)) \] 
and $F_\mathrm{nat}(x):=F^1_\text{nat}(x)$ (cf. \cite{Ghadimi2016,Nesterov2013}). Clearly, due to \cref{eq:nonexp}, if $f$ is $L$-smooth, then the function $F_\mathrm{nat}$ is Lipschitz continuous with constant $2+L$.
\section{The Stochastic Proximal Point Method}
\label{sec:stochastic_proximal_point}

\subsection{Assumptions} We first specify the basic assumptions under which we construct and study our stochastic proximal point method. Throughout this paper, we assume that the functions $f_i : \R^{m_i} \to \R$, $i \in [N]$, are continuously differentiable and $\phi : \Rn \to \Rex$ is a closed, convex, and proper mapping. Further conditions on $f$ and $\phi$ are summarized and stated below.

\begin{assumption}\label{asum:main} Let $f$ be defined as in \cref{prob:deterministic}. We assume:\
\begin{enumerate}[label=\textup{\textrm{(A\arabic*)}},topsep=0pt,itemsep=0ex,partopsep=0ex]
\item \label{A1} The functions $f_i: \R^{m_i} \to \R$ are  $L_i$-smooth and $\gamma_i$-weakly convex for all $i$.
\item \label{A2} The objective function $\psi$ is bounded from below by $\psi^\star:=\inf_x \psi(x)$.
\item \label{A3} The mapping $x \mapsto \prox{\alpha\phi}(x)$ is semismooth for all $\alpha > 0$ and all $x\in\R^n$. 
\end{enumerate}
\end{assumption}

We note that $L_i$-smoothness already ensures weak convexity of $f_i$, $i \in [N]$ (but with a potentially different constant). Let us also set $\hat{f}_i (z) := f_i(z) + ({\gamma_i}/{2})\|z\|^2$. 
We work with the following assumptions for the conjugates $\hat{f}_i^\ast$:
\begin{assumption}\label{asum:conjugate} \
\begin{enumerate}[label=\textup{\textrm{(A\arabic*)}},topsep=0pt,itemsep=0ex,partopsep=0ex,start=4]
\item \label{A4} The functions $\hat{f}_i^\ast$ are essentially differentiable (cf. \cite{Goebel2008}) with locally Lipschitz continuous gradients on the sets $\mathcal{D}_i := \mathrm{int}(\dom(\hat{f}_i^\ast)) \neq \emptyset$, $i \in [N]$.
\item \label{A5} The mappings $\nabla \hat{f}^\ast_i$ are semismooth on $\mathcal{D}_i$ for all $i$.
\end{enumerate}
\end{assumption}

By \cite[Thm.\ 18.15]{Bauschke2011}, condition \ref{A1} guarantees that the mappings $\hat{f}_i^\ast$ are $1/(L_i+\gamma_i)$-strongly convex on $\dom(\hat{f}_i^\ast)$. 
This, together with \ref{A4}, ensures that each $\hat{f}_i^\ast$ has uniformly positive definite second derivatives, i.e., there exists $\mu_\ast \geq \min_{i\in[N]} 1/(L_i+\gamma_i) > 0$ such that for all $i \in [N]$ 
\begin{align}
\label{eqn:unif_pos_def}
\iprod{h}{M_i(z)h} \geq \mu_\ast \|h\|^2  \quad \forall ~h \in \R^{m_i}, \quad \forall~M_i(z) \in \partial (\nabla \hat{f}^\ast_i)(z), \quad \forall~ z \in \mathcal{D}_i,
\end{align}
see \cite[Ex.\ 2.2]{HiriartUrruty1984}. Moreover, \ref{A4} implies that each $\hat{f}_i$ is essentially locally strongly convex \cite[Cor.\ 4.3]{Goebel2008}. Next, we state two stronger versions of \ref{A3} and \ref{A5}:\vspace{.5ex}
\begin{assumption}\label{asum:semismooth_alternative} \
\begin{enumerate}[label=\textup{\textrm{(\~{A}\arabic*)}},topsep=0pt,itemsep=0ex,partopsep=0ex,start=3]
\item \label{tA3} For every $\alpha > 0$ the proximal operator $ x \mapsto \prox{\alpha \phi}(x)$ is $\nu$-order semismooth on $\Rn$ with $0<\nu \leq 1$.
\setcounter{enumi}{4}
\item \label{tA5} The mappings $z \mapsto \nabla \hat{f}^\ast_i(z)$ are $\nu$-order semismooth on $\mathcal{D}_i$ with $0<\nu \leq 1$ for all $i$. \vspace{.5ex}
\end{enumerate}
\end{assumption}
\vspace{.5ex}

In the stochastic optimization literature, convexity and/or $L$-smoothness are standard assumptions for the component functions $f_i$ (see \cite{Defazio2014,Xiao2014,Ghadimi2016,J.Reddi2016,Asi2019}). In contrast to other recent works on stochastic proximal point methods, we neither assume convexity of $\psi$ (as in, e.g., \cite{Asi2019}) nor Lipschitz continuity of $f$ (as in, e.g., \cite{Davis2019}).\\
The additional condition \ref{A4} and the semismoothness properties \ref{A3} and \ref{A5} (or \ref{tA3} and \ref{tA5}, respectively) hold for many classical loss functions and regularizers. In fact, assumption \ref{A3} is satisfied for (group) sparse regularizations based on $\ell_1$- or $\ell_2$-norms or low rank terms such as the nuclear norm. More generally, semismoothness of the proximity operator can be guaranteed when $\phi$ is semialgebraic or tame. We refer to \cite{Bolte2009,milzarek2016numerical} for a detailed discussion of this observation. Strong semismoothness of $\prox{\alpha \phi}$ can be ensured whenever $\prox{\alpha \phi}$ is a piecewise $\mathcal{C}^2$-function (see \cite[Prop.\ 2.26]{Ulbrich2011}). For instance, if $\phi(x) = \lambda \|x\|_1$ is an $\ell_1$-regularization with $\lambda > 0$, then the associated proximity operator is the well-known soft-thresholding operator which is piecewise affine. If every mapping $\nabla \hat{f}^\ast_i$ is piecewise $\mathcal{C}^1$, then assumption \ref{A5} holds and \ref{tA5} is satisfied with $\nu=1$ if all $\nabla \hat{f}^\ast_i$, $i \in [N]$, are piecewise $\mathcal{C}^2$.

\subsection{Algorithmic Framework} We now motivate and develop our algorithmic approach in detail. 

\textit{Stochastic Proximal Point Steps.} Our core idea is to perform stochastic proximal point updates that mimic the classical proximal point iterations, \cite{Martinet1970,Martinet1972,Rockafellar1976}, for minimizing the potentially nonconvex and nonsmooth objective function $\psi$ in \cref{prob:deterministic}:
\[ x^{k+1} = \prox{\alpha_k \psi}(x^k), \]
where $\alpha_k > 0$ is a suitable step size. While $f$ is possibly nonconvex, we can conclude from \ref{A1} that $x \mapsto f_i(A_i x) + \frac{\gamma_i}{2}\|A_i(x-z)\|^2$ is a convex mapping for every $z\in\R^n$. Hence, setting $M_N:=\oneover{N} \sum_{i =1}^{N} \gamma_i \trp{A_i} A_i$, the step
\begin{align*}
\label{eqn:update_x_deterministic}
x^{k+1} &= \prox{\alpha_k\psi}^{I+\alpha_k M_N}(x^k)\\
&=\argmin_x~\psi(x) + \oneover{2N}{\sum}_{i=1}^{N}\gamma_i \|A_i(x-x^k)\|^2 + \oneover{2\alpha_k}\|x-x^k\|^2
\end{align*}
is well-defined. Utilizing \cref{eq:prox-char}, we have $p = x^{k+1}$ if and only if $p \in [x^k- \alpha_k \nabla f(p) - \alpha_k M_N(p-x^k)] - \alpha_k \partial \phi(p)$ and, hence, the proximal point update can be equivalently rewritten as the following implicit proximal gradient-type step
\begin{equation} \label{eq:ppa-imp} x^{k+1} = \prox{\alpha_k\phi}(x^k-\alpha_k \nabla f(x^{k+1}) - \alpha_k M_N(x^{k+1}-x^k)).
\end{equation}
This implicit iteration forms the conceptual basis of our method. However, as our aim is to solve the finite-sum problem \cref{prob:deterministic} in a stochastic fashion, we will use stochastic oracles to approximate the full gradient $\nabla f$ in each iteration \cite{Ghadimi2013, Ghadimi2016, Milzarek2019}. In our case, the function $f$ corresponds to an empirical expectation and thus, sampling a random subset of summands $f_i(A_i \cdot)$ can be understood as a possible stochastic oracle for $f$ and $\nabla f$. 
Specifically, for some given tuple $\mathcal S \subseteq [N]$, we can consider the following stochastic variants of $f$, $\nabla f$, and $\psi$:
\[ f_{\mathcal{S}}(x) := \oneover{|\mathcal S|}{\sum}_{i\in\mathcal S} f_{i}(A_{i} x), \quad \nabla f_{\mathcal{S}}(x) := \oneover{|\mathcal S|}{\sum}_{i\in\mathcal S} A_i^\top \nabla f_{i}(A_{i} x), \] 
and $\psi_\mathcal{S}(x) := f_\mathcal{S}(x) + \phi(x)$. Let $\mathcal S_k \subseteq [N]$ be the tuple drawn randomly at iteration $k$. The stochastic counterpart of the update \cref{eq:ppa-imp} is then obtained by replacing the gradient $\nabla f$ with the estimator $\nabla f_{\mathcal{S}_k}$ and $M_N$ with $M_{\mathcal{S}_k}:= |\mathcal S_k|^{-1} \sum_{i\in \mathcal{S}_k} \gamma_i \trp{A_i} A_i$. This yields
\begin{equation}
\label{eqn:update_x_sto}
x^{k+1} = \prox{\alpha_k \phi}(x^k-\alpha_k \nabla f_{\mathcal{S}_k}(x^{k+1}) -
{\alpha_k} M_{\mathcal S_k} (x^{k+1}-x^k)).
\end{equation}
This step can also be interpreted as a stochastic proximal point iteration 
\[ x^{k+1} = \prox{\alpha_k\psi_{\mathcal S_k}}^{I+\alpha_k M_{\mathcal S_k}}(x^k) \] 
for the sampled objective function $\psi_{\mathcal S_k}$. Consequently, the update \cref{eqn:update_x_sto} can be seen as a combination of the stochastic model-based proximal point frameworks derived in \cite{Asi2019,Davis2019} and of variable metric proximal point techniques \cite{Bonnans1995,Parente2008}. 

\textit{Incorporating Variance Reduction.} Variance reduction has proven to be a powerful tool in order to accelerate stochastic optimization algorithms \cite{Defazio2014,J.Reddi2016,Xiao2014,Hanzely2018}. Similar to \cite{J.Reddi2016}, we consider \texttt{SRVG}-type stochastic oracles that additionally incorporate the following gradient information in each iteration
\begin{align}
\label{eqn:defn_d_Sk}
v^k :=\nabla f(\tilde{x})-\nabla f_{\mathcal{S}_k}(\tilde{x}),
\end{align}
where $\tilde{x} \in \Rn$ is a reference point that is generated in an outer loop. 
This leads to stochastic proximal point-type updates of the form
\begin{equation}
\label{eqn:update_x_stochastic}
x^{k+1} = \prox{\alpha_k \phi}(x^k-\alpha_k [\nabla f_{\mathcal{S}_k}(x^{k+1}) + v^{k}] -
{\alpha_k} M_{\mathcal S_k} (x^{k+1}-x^k)).
\end{equation}
Our subsequent analysis and discussion focuses on this general formulation.

\textit{An Implementable Strategy for Performing the Implicit Step \cref{eqn:update_x_stochastic}}. We now introduce an alternative equation-based characterization of the implicit update \cref{eqn:update_x_stochastic}. Let us set $b_k := |\mathcal S_k|$ and let $(\ixmap_k(1),\dots,\ixmap_k(b_k))$
enumerate the elements of the tuple $\mathcal S_k$. We will often
abbreviate $\ixmap_k$ by $\ixmap$.
We now define $\xi^{k+1} = (\xi_1^{k+1}, \dots, \xi_{b_k}^{k+1})$ by
\begin{equation}\label{eqn:update_xi_stochastic}
\xi_i^{k+1} := \nabla \hat{f}_{\ixmap(i)}(A_{\ixmap(i)} x^{k+1}) = \nabla f_{\ixmap(i)}(A_{\ixmap(i)} x^{k+1}) + \gamma_{\ixmap(i)} A_{\ixmap(i)} x^{k+1} \quad i \in [b_k].
\end{equation}
Under assumption \ref{A1}, \cite[Thm.\ 4.20]{Beck2017} yields
\begin{displaymath}
\xi_i^{k+1} = \nabla \hat{f}_{\ixmap(i)}(A_{\ixmap(i)} x^{k+1}) \quad \iff \quad \nabla \hat{f}^\ast_{\ixmap(i)}(\xi_i^{k+1}) = A_{\ixmap(i)} x^{k+1}.
\end{displaymath}
Thus, setting $\hat v^k := \alpha_k (v^k-M_{\mathcal S_k}x^k)$, the step \cref{eqn:update_x_stochastic} is equivalent to the system
\begin{equation} \label{eqn:nonlinear_system_stochastic} \left[ \begin{array}{l} x^{k+1} = \prox{\alpha_k \phi}\left(x^k- \frac{\alpha_k}{b_k} \sum_{i =1}^{b_k} \trp{A_{\ixmap(i)}}\xi_i^{k+1}-\hat{v}^k\right), \\[1ex]
\nabla \hat{f}_{\ixmap(i)}^\ast(\xi_i^{k+1}) = A_{\ixmap(i)}\prox{\alpha_k \phi}\left(x^k-\frac{\alpha_k}{b_k}\sum_{i =1}^{b_k} \trp{A}_{\ixmap(i)}\xi_i^{k+1} - \hat{v}^k \right) \quad \forall~i \in [b_k]. \end{array} \right. 
\end{equation}
Similar to \cite{Zhao2010,Li2018}, we use a semismooth Newton method to solve the system of nonsmooth equations defining $\xi^{k+1}$ in an efficient way. Importantly, the dimension of this system and of $\xi^{k+1}$ is controlled by the batch size $b_k$ which is an advantage if $b_k\ll n$. We allow approximate solutions of the system \cref{eqn:nonlinear_system_stochastic} which results in inexact proximal steps. This potential inexactness is an important component of our algorithmic design and convergence analysis that has not been considered in other stochastic proximal point methods \cite{Asi2019,Davis2019}.
The full method is shown in \cref{alg:snspp}. The semismooth Newton method for \cref{eqn:nonlinear_system_stochastic} is specified and discussed in the next section. In this article, we primarily focus on the variance-reduced update \cref{eqn:update_x_stochastic}, yet the technique and results presented in \cref{sec:subproblem_newton} also hold true for the general update \cref{eqn:update_x_sto}.

\textit{The Full Stochastic Setup and Stochastic Assumptions.} We now formally specify the notion of admissible stochastic oracles for our problem.
\begin{definition}
\label{def:admissible_sampling}
Let $\mathcal{S} \sim \mathbb{P}$ be a $b$-tuple of elements of
$[N]$, where $b\in[N]$ is fixed. Let $\ixmap(i)\in[N]$ denote
the random number in the $i$-th position of $\mathcal{S}$.
We call $\mathcal{S} \sim \mathbb{P}$ an admissible sampling procedure if,
for all $z_i \in \R^\ell$, $i\in [N]$, $\ell\in\mathbb{N}$, it holds that
$\mathbb{E}_\mathbb{P} [z_\mathcal{S}] = \oneover{N}\sum_{i=1}^{N} z_i$
where $z_\mathcal{S} := \oneover{b} \sum_{i=1}^{b} z_{\ixmap(i)}$.
\end{definition}
If $\mathcal{S}\sim \mathbb{P}$ is an admissible sampling procedure, then we have $\E_\mathbb{P}[f_{\mathcal{S}}(x)] = f(x)$ and $\mathbb{E}_\mathbb{P} [\nabla f_\mathcal{S}(x)] = \nabla f(x)$ for all $x \in \Rn$.
%
In the simplest case, we can choose $\mathcal{S}$ by drawing $b$ elements from $[N]$ under a uniform distribution (cf.\ \cite{Xiao2014} for a similar setting). This is an admissible sampling procedure in the sense of \cref{def:admissible_sampling}, regardless of whether we draw with or without replacement (cf.\ \cite[\S2.8]{Lohr2010}). 
\begin{algorithm}[t]
\caption{\texttt{SNSPP}}
\label{alg:snspp}
\begin{algorithmic}
\REQUIRE $\tilde{x}^0 \in \R^n$, $m, S \in \mathbb{N}$, and,
for $s=0,\ldots,S$, $k=0,\ldots,m-1$,
step sizes $\alpha_{k}^{s}>0$, batch sizes $b_{k}^{s}\in\N$,
and tolerances $\epsilon_{k}^{s}\ge 0$.
\FOR{$s=0,1,2,\dots,S$}
\STATE Set $x^0 := x^{s,0} := \tilde{x} := \tilde{x}^{s}$, and, for $0\le k<m$,
$\alpha_{k}:=\alpha_{k}^{s}$, $b_{k}:=b_{k}^{s}$, and $\epsilon_{k}:=\epsilon_{k}^{s}$.
\FOR{$k = 0,1,2,\dots,m-1$}
\STATE \textbf{(Sampling)} Sample $\mathcal{S}_k=\mathcal{S}_{k}^s$ with $|\mathcal{S}_k|=b_k$ and set \vspace{0.5ex}
\begin{center} $v^{k}:=v^{s,k} := \nabla f(\tilde{x}) - \nabla f_{\mathcal{S}_k}(\tilde{x})$,\quad $\hat{v}^k := \hat{v}^{s,k} := \alpha_{k} (v^k-M_{\mathcal S_k}x^k)$. \vspace{0.5ex} \end{center}
\STATE \textbf{(Solve subproblem)} Compute $\xi^{k+1}=\xi^{s,k+1}$ by invoking \cref{alg:semismooth_newton} with input $x^k, \alpha_{k}, \mathcal{S}_k, -\hat{v}^k$ and $\epsilon_{k}$.
\STATE \textbf{(Update)} Set $x^{k+1}:=x^{s,k+1} := \prox{\alpha_{k} \phi}\left(x^k -
\frac{\alpha_{k}}{b_{k}}\sum_{i=1}^{b_{k}} \trp{A}_{\ixmap(i)} \xi_i^{k+1}-\hat{v}^k\right)$.
\ENDFOR
\STATE Option I: Set $\tilde{x}^{s+1} := x^{m}$.
\STATE Option II: Set $\tilde{x}^{s+1} := \frac{1}{m} \sum_{k=1}^{m} x^k$.
\ENDFOR
\RETURN $\tilde{x}^{S+1}$; $x_{\pi}$ drawn uniformly from
$\{x^{s,k}\}_{0\le k<m}^{0\le s\le S}$.
\end{algorithmic}
\end{algorithm}
\begin{remark}
Note that $x^k$, $\alpha_k$, etc., serve as abbreviations when the value of
$s$ is clear. The notation $x^{s,k}$, $\alpha_k^s$, etc., can be used
to highlight the full $(s,k)$-dependence.
\end{remark}

\section{A Semismooth Newton Method for Solving the Subproblem} \label{sec:subproblem_newton}
In the following, we assume that we are given a $b$-tuple
$\mathcal{S}=(\ixmap(1),\ldots,\ixmap(b))$ of elements of $[N]$,
a step size $\alpha >0$, and vectors $x,v\in \R^n$. Let $m_{\mathcal{S}} := \sum_{i=1}^{b}m_{\ixmap(i)}$ denote the dimension of the subproblem and let $\mathcal D = \prod_{i \in \mathcal S} \mathcal D_i \subseteq \R^{m_{\mathcal{S}}}$. Now, the second line in \cref{eqn:nonlinear_system_stochastic} corresponds to a system of nonlinear equations which can be reformulated as
\begin{equation}
\label{eqn:newton_equation}
\mathcal{V}(\xi) = 0, 
\end{equation}
where $\mathcal{V}: \mathcal{D} \to \R^{m_{\mathcal{S}}}$, $\mathcal{V}(\xi) :=\trp{\trp{(\mathcal{V}_1(\xi)}, \dots, \trp{\mathcal{V}_{b}(\xi)})}$, and $\xi := \trp{(\trp{\xi}_1, \dots, \trp{\xi}_{b})}$. Setting $\mathcal{A}_\mathcal{S} := \oneover{b} \trp{(\trp{A}_{\ixmap(1)},\ldots,\trp{A}_{\ixmap(b)})} \in \R^{m_\mathcal{S}\times n}$, each $\mathcal{V}_i$ is defined via
\begin{equation}
\label{eqn:definition_V}
\mathcal{V}_i : \mathcal{D} \to \R^{m_{\ixmap(i)}}, \quad \mathcal{V}_i(\xi) = \nabla \hat{f}_{\ixmap(i)}^\ast(\xi_i) - A_{\ixmap(i)} \prox{\alpha
\phi}\left(x - \alpha\mathcal{A}_\mathcal{S}^\top \xi +v \right).
\end{equation}
The Newton step of this system is given by 
\begin{equation} \label{eq:sn-step}
\mathcal{W}(\xi)d = - \mathcal{V}(\xi),
\end{equation}
where $\mathcal{W}(\xi) \in \hat{\partial} \mathcal{V}(\xi)$ is an element of the surrogate differential $\hat{\partial} \mathcal{V}$ defined via
\begin{align}
\hat{\partial} \mathcal{V}(\xi) := \Big\{&\mathrm{Diag}\left(H_i(\xi_i)_{i = 1, \dots, b}\right) + \alpha b \mathcal{A}_\mathcal{S}  U(\xi) \trp{\mathcal{A}_\mathcal{S}}
~~\big \vert \nonumber\\
&\quad U(\xi) \in \partial \prox{\alpha \phi}\bigl(x - \alpha
\trp{\mathcal{A}}_\mathcal{S}\xi +v \bigr),\; H_i(\xi_i) \in \partial(\nabla \hat{f}_{\ixmap(i)}^\ast)(\xi_i)
\; \forall~i \in [b] \Big\}. \nonumber
\end{align}
We first present several basic properties of the operators and functions involved in the Newton step \cref{eq:sn-step}. The nonexpansiveness of the proximity operator \cref{eq:nonexp} and \cite[Prop.\ 2.3]{Jiang1995} imply the next result (see also \cite[Lem.\ 3.3.5]{milzarek2016numerical}). 

\begin{proposition}
\label{prop:prox_phi_semidef}
Let $\alpha >0$ and $x \in \R^n$ be given. Each element $U \in \partial \prox{\alpha \phi}(x)$ is a symmetric and positive semidefinite $n \times n$ matrix.
\end{proposition}

\begin{proposition}
Suppose that the conditions \ref{A3}, \ref{A4}, and \ref{A5} are satisfied. Let $\mathcal{S}$ be a $b$-tuple of elements of $[N]$, and let $\alpha>0$ and $x,v\in\R^n$ be given. Then, the function $\mathcal{V}$ is semismooth on $\mathcal{D}$ w.r.t. $\hat{\partial} \mathcal{V}$. If \ref{tA3} and \ref{tA5} hold instead of \ref{A3} and \ref{A5}, $\mathcal{V}$ is $\nu$-order semismooth w.r.t. $\hat{\partial} \mathcal{V}$. 
\end{proposition}
\begin{proof}
The first claim follows using the chain rule for semismooth functions \cite[Thm.\ 7.5.17]{Facchinei2003}. If we assume \ref{tA3} and \ref{tA5} instead, the claim follows from \cite[Prop.\ 3.8]{Ulbrich2011} and \cite[Prop.\ 3.6]{Shapiro1990}.
\end{proof}
In the following, we show that the function $\mathcal{V}$ can be interpreted as a gradient mapping. Thus, finding a root of $\mathcal{V}$ is equivalent to finding a stationary point.
\begin{proposition}
\label{prop:derive_U}
Let the assumptions \ref{A1} and \ref{A4} hold. Let $\alpha >0$ and $x,d \in \R^n$ be given and let $\mathcal{S}$ be a $b$-tuple of elements of $[N]$. For $\xi \in \R^{m_\mathcal{S}}$, we define
\[ \mathcal{U}(\xi) := {\sum}_{i=1}^{b} \hat{f}_{\ixmap(i)}^\ast(\xi_i) + \frac{b}{2\alpha} \|z(\xi)\|^2 - \frac{b}{\alpha} \mathrm{env}_{\alpha \phi}(z(\xi)), \quad z(\xi) := x - \alpha \mathcal A_{\mathcal S}^\top \xi + v. \]
Then, $\mathcal{U}$ is $\mu_\ast$-strongly convex on the set $\mathcal E := \prod_{i=1}^b \dom(\hat{f}_{\ixmap(i)}^\ast)$ and we have $\nabla \mathcal{U}(\xi) = \mathcal{V}(\xi)$ for all $\xi \in \mathcal{D}$ where $\mathcal{V}$ is defined in \cref{eqn:definition_V}.
\end{proposition}
\begin{proof}
For every $\xi \in \mathcal{D}$ and $i\in [b]$, we have $\frac{\partial z}{\partial\xi_i}(\xi) = -\frac{\alpha}{b}\trp{A}_{\ixmap(i)}$ and
\begin{align*}
\nabla_{\xi_i} \mathcal{U}(\xi) &= 
\nabla \hat{f}_{\ixmap(i)}^\ast(\xi_i) 
+ \frac{b}{\alpha} \trp{\frac{\partial z}{\partial\xi_i}(\xi)}(z(\xi) - (z(\xi) - \prox{\alpha \phi}(z(\xi)))) \\
&= \nabla \hat{f}_{\ixmap(i)}^\ast(\xi_i) - A_{\ixmap(i)}
\prox{\alpha \phi}(z(\xi)) = \mathcal{V}_i(\xi),
\end{align*}
where we used \cref{eqn:moreau_gradient}. For the first statement, note that \ref{A1} implies strong convexity of $\hat{f}_i^\ast$ on $\dom(\hat{f}_{i}^\ast)$ for $i=1,\dots,N$. Applying Moreau's identity, \cite[Thm.\ 6.67]{Beck2017},
\begin{equation}\label{enveq}
\onehalf \|z\|^2 - \mathrm{env}_{\alpha \phi}(z) = \alpha^2\mathrm{env}_{\alpha^{-1}\phi^\ast}(z/\alpha),
\end{equation}
we can use the fact that the Moreau envelope of a proper,
closed, and convex function is convex \cite[Thm.\ 6.55]{Beck2017}. Hence, the mapping $\xi \mapsto \frac{b}{2\alpha} \|z(\xi)\|^2 - \frac{b}{\alpha} \mathrm{env}_{\alpha \phi}(z(\xi))$
is convex as $\xi \mapsto z(\xi)$ is affine. Altogether, $\mathcal{U}$ is $\mu_\ast$-strongly convex on $\mathcal{E}$ (cf.\ \cref{eqn:unif_pos_def}).
\end{proof}
In \cref{alg:semismooth_newton}, we formulate a globalized semismooth Newton method for solving the nonsmooth system \cref{eqn:newton_equation}. Specifically, the result in \cref{prop:derive_U} enables us to measure descent properties of a semismooth Newton step using $\mathcal U$ and to apply Armijo line search-based globalization techniques. Based on the results on  SC${}^{1}$ minimization (cf. \cite{Facchinei1995,Zhao2010,Ulbrich2011}), we obtain the following convergence result.

\begin{algorithm}[t]
\caption{Semismooth Newton Method for Solving \cref{eqn:newton_equation}}
\label{alg:semismooth_newton}
\begin{algorithmic}
\REQUIRE $x, v \in \R^n$, $\alpha > 0$, a $b$-tuple $\mathcal{S}$ of elements of $[N]$, and a tolerance $\epsilon_{\mathrm{sub}}$.\\
Choose an initial point $\xi^0$ such that $\xi^0_i \in \mathcal{D}_i$ for all $i=1,\dots,b$. Choose parameters $\hat{\gamma} \in (0, \onehalf)$, $\eta \in (0, 1)$, $\rho \in (0, 1)$, $\tau \in (0,1]$, and $\tau_1, \tau_2 \in (0,1)$. Set $j=0$.
\WHILE{$\|\nabla\mathcal{U}(\xi^{j})\| > \epsilon_{\mathrm{sub}}$}
\STATE \textbf{(Newton direction)} Choose $\mathcal{W} \in \hat{\partial}
\mathcal{V}(\xi^j)$, set $\eta_j := \tau_1 \min\{\tau_2,
\|\mathcal{V}(\xi^j)\|\}$, and approximately solve the linear system
\begin{displaymath}
(\mathcal{W} + \eta_j I)d^j = -\mathcal{V}(\xi^j) 
\end{displaymath}
 via the conjugate gradient method such that $\|r^j\| \leq \min\{\eta, \|\mathcal{V}(\xi^j)\|^{1+\tau}\}$ with $r^j := (\mathcal{W} + \eta_j I)d^j + \mathcal{V}(\xi^j)$.\\
\STATE \textbf{(Armijo line search)} Find the smallest non-negative integer $\ell_j$ such that
\begin{displaymath}
\mathcal{U}(\xi^j + \rho^{\ell_j} d^j) \leq \mathcal{U}(\xi^j) + \hat{\gamma}\rho^{\ell_j} \langle \nabla \mathcal{U}(\xi^j), d^j \rangle
\end{displaymath}
and $\xi^j_i+\rho^{\ell_j}d_i^j \in \mathcal D_i$ for all $i=1,\dots,b$. Set $\beta_j := \rho^{\ell_j}$.\\
\STATE \textbf{(Update) }Compute the new iterate $\xi^{j+1} = \xi^j + \beta_j d^j$ and set $j \leftarrow j+1$.
\ENDWHILE
\RETURN $\xi^{j}$
\end{algorithmic}
\end{algorithm}

\begin{theorem}
\label{thm:ssn_convergence}
Let the assumptions \ref{A1}--\ref{A5} be satisfied and let the sequence $\{\xi^j\}$ be generated by \cref{alg:semismooth_newton}. Then, $\{\xi^j\}$ converges q-superlinearly to the unique solution $\hat\xi \in \mathcal D$ of \cref{eqn:newton_equation}, i.e.,
$\|\xi^{j+1} - \hat\xi\| = o(\|\xi^{j} - \hat\xi\|)$
as $j \to \infty$.
Moreover, under \ref{tA3} and \ref{tA5}, we obtain
\begin{displaymath}
\|\xi^{j+1} - \hat\xi\| = \mathcal{O}(\|\xi^{j} - \hat\xi\|^{1+\min\{\tau, \nu\}}) \quad \text{for all $j$ sufficiently large}.
\end{displaymath}
\end{theorem}
\begin{proof}
By construction, we have $\{\xi^j\} \subset \mathcal{D}$ and the set $\mathcal{D}$ is open. \cref{prop:derive_U} and \ref{A4} imply that $\mathcal U$ is strongly convex (on $\mathcal E$) and essentially differentiable. Hence, $\mathcal{U}$ has a unique minimizer $\hat\xi\in\mathcal{D}$ which is also the unique solution of \cref{eqn:newton_equation}. For every $\xi \in \mathcal{D}$, the matrices $\mathcal W(\xi) \in \hat{\partial} \mathcal{V}(\xi)$ are positive definite by \cref{eqn:unif_pos_def} and \cref{prop:prox_phi_semidef}. Using standard arguments (see \cite[Thm.\ 3.4]{Zhao2010} and \cite[Thm.\ 3.6]{Li2018}), it can be shown that the sequence $\{\xi^j\}$ generated by \cref{alg:semismooth_newton} converges to $\hat\xi$. Under \ref{A1}--\ref{A5}, we conclude from equation (67) in the proof of \cite[Thm.\ 3.5]{Zhao2010} that
$\|\xi^j + d^j - \hat\xi\| \leq o(\|\xi^{j} - \hat\xi\|)$
holds for all $j$ sufficiently large.
If assumptions \ref{tA3} and \ref{tA5} are satisfied instead of \ref{A3} and \ref{A5}, then we have $\|\xi^j + d^j - \hat\xi\| \leq \mathcal{O}(\|\xi^{j} - \hat\xi\|^{1+\min\{\tau, \nu\}})$. Finally, let us show that in a neighborhood of the limit point the unit step size is accepted by the Armijo line search. Setting $\tilde{\mathcal{W}}_j:= \mathcal{W} + \eta_j I$ and using $\mathcal V(\xi^j) \to 0$, we can infer
\begin{align*}
\|d^j\| = \|\tilde{\mathcal{W}}_j^{-1} (r^j - \mathcal{V}(\xi^j))\| \leq  \|\tilde{\mathcal{W}}_j^{-1}\|(\|r^j\| + \|\mathcal{V}(\xi^j)\|) \leq 2 \lambda_{\min}(\tilde{\mathcal{W}_j})^{-1} \|\mathcal{V}(\xi^j)\|,
\end{align*}
for all $j$ sufficiently large. Thus, we have
\begin{align*}
- \frac{\langle \nabla \mathcal{U}(\xi^j), d^j \rangle}{\|d^j\|^2} \geq \frac{\lambda_{\min}(\tilde{\mathcal{W}_j})^2}{4} \frac{\langle -\nabla \mathcal{U}(\xi^j), d^j \rangle}{\|\nabla \mathcal{U}(\xi^j)\|^2} \geq \frac{\lambda_{\min}(\tilde{\mathcal{W}_j})^2}{4\lambda_{\max}(\tilde{\mathcal{W}_j})},
\end{align*}
where the second inequality comes from \cite[Prop.\ 3.3]{Zhao2010}. 
Due to strong convexity, there exists $\tilde \rho > 0$ such that $\frac{\lambda_{\min}(\tilde{\mathcal{W}_j})^2}{4\lambda_{\max}(\tilde{\mathcal{W}_j})}\geq \tilde \rho > 0$ for all $j$. Thanks to \cite[Thm.\ 3.3]{Facchinei1995}, $\beta_j = 1$ then fulfills the Armijo condition for $j$ sufficiently large which concludes the proof.
\end{proof}
\section{Controlling the Inexactness of the Update}
\label{sec:inexactness}
In this section, we will discuss the stopping criterion of \cref{alg:semismooth_newton}. Let $x \in \R^n$, $\alpha >0$, and a tuple $\mathcal{S}$ of elements of $[N]$ be given. By \cref{prop:derive_U}, $\mathcal{U}$ is $\mu_\ast$-strongly convex on $\mathcal D \subset \mathcal E$. Thus, the gradient $\nabla \mathcal U$ is a $\mu_\ast$-strongly monotone operator on $\mathcal D$. 
Let $\hat{\xi} := \argmin_{\xi} \mathcal{U}(\xi) \in \mathcal{D}$ again denote the unique minimizer of $\mathcal U$. Then, we have 
\begin{align}
\label{eqn:inexactness_xi}
\|\xi - \hat{\xi} \| \leq {\mu_\ast}^{-1} \|\nabla\mathcal{U}(\xi)\| \quad \forall ~ \xi \in \mathcal D,
\end{align}
Hence, the stopping criterion of \cref{alg:semismooth_newton} -- $\|\nabla\mathcal{U}(\xi^{j})\| \leq \epsilon_{\mathrm{sub}}$ -- allows to control the error $\|\xi^{j} - \hat{\xi} \|$. As we solve each subproblem inexactly, the updates $(x^{k+1}, \xi^{k+1})$ in \cref{alg:snspp} are not an exact solution to \cref{eqn:update_x_stochastic}. It is desirable to control the error $\|x^{k+1} - \hat{x}^{k+1}\|$ 
in each iteration, where $\hat{x}^{k+1}$ is the exact solution of \cref{eqn:update_x_stochastic}.
This is addressed in the following result.
\begin{proposition}
\label{prop:inexactness_bound}
Let us define $\bar{A}:= \max_{i \in [N]} \|A_i\|$ and let $x,v\in\R^n$, $\alpha>0$, and a $b$-tuple $\mathcal{S}$ of elements of $[N]$ be given. Suppose that $\mathcal{U}$, defined in \cref{prop:derive_U}, is $\mu_\ast$-strongly convex on $\mathcal{E}$ and let $\hat{\xi} = \argmin_{\xi} \mathcal{U}(\xi) \in \mathcal D$ be the unique minimizer of $\mathcal U$. Suppose that \cref{alg:semismooth_newton} -- run with tolerance $\epsilon_{\mathrm{sub}}$ -- returns $\xi$. Then, setting
\begin{align*}
\hat{x}^+ := \prox{\alpha \phi}(x - \alpha \trp{\mathcal{A}_{\mathcal{S}}} \hat{\xi}+v) \quad \text{and} \quad
x^+:= \prox{\alpha \phi}(x - \alpha \trp{\mathcal{A}_{\mathcal{S}}} \xi  +v ),
\end{align*}
it holds that
$\|\xi-\hat{\xi}\|\leq \frac{\epsilon_{\mathrm{sub}}}{\mu_\ast}$ and  $\|x^+ - \hat{x}^+\| \leq \alpha \| \trp{A_{\mathcal{S}}} (\xi - \hat{\xi})\|\leq \frac{\alpha \bar{A}}{\mu_\ast\sqrt{b}} \epsilon_{\mathrm{sub}}.$
\end{proposition}
\begin{proof}
The bound on $\|\xi-\hat{\xi}\|$ follows directly from  \cref{eqn:inexactness_xi}. Utilizing the nonexpansiveness of the proximity operator, we can estimate
\begin{align*}
\|x^+ - \hat{x}^+\| &= \|\prox{\alpha \phi}\big(x - \alpha \trp{\mathcal{A}_{\mathcal{S}}} \xi  +v \big)  - \prox{\alpha \phi}\big(x - 
\alpha \trp{\mathcal{A}_{\mathcal{S}}} \hat{\xi}+v \big)\|  \\
& \leq \alpha \|  \trp{A_{\mathcal{S}}} (\xi - \hat{\xi})\| \leq \frac{\alpha \bar{A}}{\sqrt{b}} \|\xi -\hat{\xi} \| \leq \frac{\alpha \bar{A}}{\mu_\ast\sqrt{b}} \epsilon_{\mathrm{sub}}.  
\end{align*}
\end{proof}
\section{Convergence Analysis}
\label{sec:convergence}
We first introduce several important constants. Let \ref{A1} of \cref{asum:main} be satisfied. We denote the Lipschitz constant of $\nabla f$ by $L$ and we set ${\mL} := \max_i L_i \|A_i\|^2$.
For $b \in [N]$, let us define $\bar{L}_b := \max_{\mathcal{S}, |\mathcal{S}|=b} L_\mathcal{S}$ where $L_\mathcal{S}$ is the Lipschitz constant of $\nabla f_\mathcal{S}$, and $M_{\mathcal{S}} := \oneover{b}\sum_{i\in \mathcal{S}} \gamma_i \trp{A}_iA_i$. Then, for any tuple $\mathcal{S}$ it holds that
\begin{align}
\label{eqn:definition_of_constants}
L \leq \frac1N {\sum}_{i=1}^N L_i\|A_i\|^2\leq {\mL},\qquad
\bar{L}_b \leq {\mL},\qquad
\|M_{\mathcal{S}}\| \leq \max_{i = 1, \dots, N} \gamma_i \cdot \bar{A}^2 =: \bar M.
\end{align}
We now formulate the main convergence results for \cref{alg:snspp}. For simplicity, we assume that $\mathcal{S}_k$ is drawn uniformly from $[N]$ with replacement for all $s$ and $k$.\footnote{Without replacement, only some of the constants change, see \cref{cor:c1} for details.} 
\subsection{Weakly Convex Case}
\begin{theorem}
\label{thm:general_case}
Let the iterates $\{x^{s,k}\}$ be generated by \cref{alg:snspp} with
$S=\infty$, constant batch sizes $b_k^s=b$, and using Option I. Let
\cref{asum:main} and \cref{asum:conjugate} be satisfied and assume
\begin{equation} \sum_{s=0}^\infty\sum_{k=0}^{m-1} \alpha_k^s = \infty,  \quad \sum_{s=0}^\infty\sum_{k=0}^{m-1}
\alpha_k^s (\epsilon_k^s)^2 < \infty, \quad \alpha_k^s \leq \min\{1,\hat \alpha\} \quad \forall~s, k, \end{equation}
where $\hat \alpha := \bar\eta \max\{2L+\bar M, [1+m/\sqrt{2b}]\mL + \max\{L,\bar M\}\}^{-1}$ and $\bar\eta \in (0,1)$.
Then, for all $0\le k<m$,
$\{\E\|F_{\mathrm{nat}}(x^{s,k})\|\}_{s\in\N_0}$ converges to zero
and $\{F_{\mathrm{nat}}(x^{s,k})\}_{s\in \N_0}$ converges to zero
almost surely as $s\to\infty$.
\end{theorem}
Similar to \cite[Thm.\ 1]{J.Reddi2016}, we obtain the following rate of convergence.
\begin{corollary}
\label{cor:general_case_rate}
Let the iterates $\{x^{s,k}\}$ be generated by \cref{alg:snspp} with
$S \in \N$, constant step sizes $\alpha_k^s=\alpha$,
constant batch sizes $b_k^s=b$, and using Option I.
Let \cref{asum:main} and \cref{asum:conjugate} be satisfied.
Let $\bar\eta \in (0,1)$ be given and assume $\alpha \leq \hat{\alpha}$,
where $\hat{\alpha}$ is defined as in \cref{thm:general_case}.
Then, it holds that
\begin{align*}
\E\|F^\alpha_{\mathrm{nat}}(x_\pi)\|^2
\le \frac{2\alpha[\psi(\tilde x^0)-\psi^\star+ \alpha \cdot \mathcal O(\sum_{s=0}^S\sum_{k=0}^{m-1} (\epsilon_k^s)^2)]} {(1-\bar{\eta})^3 \cdot m(S+1)}.
\end{align*}
\end{corollary}
The proofs are given in \cref{proof:thm:general_case}.
%
\subsection{Strongly Convex Case}
In this section, we establish q-linear convergence of \cref{alg:snspp} if $\psi$ is strongly convex. We derive -- similar to Thm.\ 3.1 in \cite{Xiao2014} -- convergence in terms of the objective function if we assume each $f_i$ to be convex. We present an additional result for weakly convex $f$ and strongly convex $\phi$. The proofs are given in \cref{proofs:str_convex}.
In the following, we suppose that in iteration $s$ of the outer and iteration $k$ of the inner loop of \cref{alg:snspp}, the tolerances $\epsilon_k^s$ satisfy the bound 
\begin{align}
\label{eqn:eps_cond_strconvex}
\epsilon_{k}^s \leq \delta_s \|F_{\mathrm{nat}}(\tilde x^s)\| 
\end{align}
for all $k \in \{0,\dots,m-1\}$, $s \in \N$, and for some sequence
$\R_{++} \ni \delta_s \to 0$. Notice, since $\nabla f(\tilde x^s)$ is known, $\|F_{\mathrm{nat}}(\tilde x^s)\|$ can be computed without additional costs.
\begin{theorem} \label{thm:str_convex_case_obj}
Let \cref{asum:main} and \cref{asum:conjugate} be satisfied and suppose that each function $f_i$ is convex and $\psi=f+\phi$ is $\mu$-strongly convex with $\mu>0$. 
Consider \cref{alg:snspp} with $S=\infty$ and Option II using constant step sizes
$\alpha_k^s=\alpha>0$, and constant batch sizes $b_k^s=b$. For $\theta \in(0,1/2)$, let the step size $\alpha$ satisfy
\begin{align}
\label{eqn:strcon_obj_alpha_assumption}
\alpha \leq \left[L+\tfrac{{\mL}}{b}(\tfrac{4}{1-2\theta}+3)\right]^{-1}
\end{align}  
and let condition \cref{eqn:eps_cond_strconvex} hold for a given sequence $\{\delta_s\}$. If $\delta_s$, $s \in \N$, is sufficiently small and the inner loop length $m$ sufficiently large, then $\{\psi(\tilde{x}^s)\}$ converges q-linearly in expectation to $\psi^\star$ with rate at least $1-\theta$, i.e.,
for all $s$, we have
\begin{align*} \E[\psi(\tilde{x}^{s+1})-\psi^\star] \leq (1-\theta) \E[\psi(\tilde{x}^{s})-\psi^\star]. 
\end{align*}
\end{theorem}

More formal and explicit conditions on $\delta_s$ and $m$ can be found in the proof of \cref{thm:str_convex_case_obj} in \cref{proofs:str_convex}.
\begin{theorem} \label{thm:str_convex_case}
Let \cref{asum:main} and \cref{asum:conjugate} be satisfied and let $\phi$ and $\psi=f+\phi$ be $\mu_\phi$- and $\mu$-strongly convex, respectively, with 
\[ \mu_\phi I - M_N \succeq \mu I. \]
Consider \cref{alg:snspp} with $S=\infty$ and Option I, using constant step sizes $\alpha_k^s=\alpha>0$ and constant batch sizes $b_k^s=b$. Assume that $\alpha \leq [L+\sqrt{2/b}\cdot m\mL]^{-1}$ and let \cref{eqn:eps_cond_strconvex} hold for a given sequence $\{\delta_s\}$ satisfying $\delta_s < \min\{\frac{2\alpha \mu}{1+\alpha\mu_\phi}, \frac{1+\alpha(\mu+\mu_\phi)}{1+\alpha\mu_\phi}\}$ for all $s$. 
%
%
Then, the iterates $\{\tilde x^s\}$ converge q-linearly in expectation to the unique solution $x^\star$ of problem \cref{prob:deterministic}, i.e.,
as $s \to \infty$, we have
\begin{align*} \E\|\tilde x^{s+1}-x^\star\|^2 \leq \left[ 1- \tfrac{2\alpha\mu}{1+\alpha(\mu_\phi+\mu)} + \mathcal O(\delta_s) \right]  \E\|\tilde x^{s}-x^\star\|^2. \end{align*}
\end{theorem}

\section{Numerical Experiments}
\label{sec:numeric}
In the following, we investigate the practical performance of \cref{alg:snspp} which we will refer to as \texttt{SNSPP}. Specifically, we compare \texttt{SNSPP} with three state-of-the-art and benchmark stochastic algorithms -- namely \texttt{SVRG}, \texttt{SAGA}, and \texttt{AdaGrad} for all of which we use their proximal versions. 
\begin{itemize}
\item \texttt{SVRG} \cite{Johnson2013,Xiao2014}\footnote{We use the acronym \texttt{SVRG} even though the method is named \texttt{Prox-SVRG} in the original article.} is a (minibatch) stochastic gradient method with variance reduction. Its proximal version was proposed in \cite{Xiao2014} for convex composite problems and extended to nonconvex objectives in \cite{J.Reddi2016}.
\item \texttt{SAGA} \cite{Defazio2014} is a stochastic proximal gradient method utilizing a different variance reduction strategy. It has been analyzed for the nonconvex case in \cite{J.Reddi2016}. 
\item \texttt{AdaGrad} \cite{Duchi2011} is a proximal adaptive gradient method that assigns larger step sizes to features which have been rarely explored up to the current iteration.
\end{itemize}

\subsection{General Setting}
\label{sec:numeric_params}
For all experiments, we set the parameter $m$ in \cref{alg:snspp} to $m=10$. For \cref{alg:semismooth_newton}, we use $\hat{\gamma} = 0.4$, $\eta = 10^{-5}$, $\rho = 0.5$, $\tau = 0.9$, $\tau_1 =0.5$, $\tau_2 = 2\cdot10^{-4}$ and terminate if $\|\nabla \mathcal{U}(\xi^{j})\| \leq 10^{-3}$. Typically, \cref{alg:semismooth_newton} reaches the desired accuracy within less than $10$ iterations. We also ran the experiments with the adaptive accuracy $\epsilon_{k}^s = \delta_s\|F_{\mathrm{nat}}(\tilde x^s)\|$, $\delta_s = 0.1$, introduced in \cref{eqn:eps_cond_strconvex} but did not observe any significant effects on the results.
We precompute an estimate for $\psi^\star$ by running a solver for a large number of iterations: for the experiments in \cref{sec:experiments_logreg}, we use \texttt{scikit-learn} for this step; for the experiments in \cref{sec:experiments_student} we use our own implementation of \texttt{SAGA}.
We use constant batch and step sizes throughout all experiments.
The experiments were performed in \texttt{Python 3.8}.\footnote{Code is available at \url{https://github.com/fabian-sp/snspp}.}
\subsection{Logistic Regression with \boldmath$\ell_1$-Regularization}
\label{sec:experiments_logreg}
Sparse logistic regression is a classical model for binary classification \cite{Koh2007, Hastie2009}. Let the coefficient matrix $A = (a_1^\top,\dots,a_N^\top)^\top \in \R^{N\times n}$ and the binary labels $b_i \in \{-1,1\},~ i=1,\dots,N$ be given. The associated sparse regression problem can then be formulated as follows:
\begin{align*}
\min_{x}~\oneover{N}\sum_{i=1}^{N} \ln(1+\exp({-b_i\langle a_i, x\rangle})) + \lambda \|x\|_1, \quad \lambda > 0.
\end{align*}
This problem is of the form \cref{prob:deterministic} with $m_i = 1 $, $A_i = b_i a_i$, and $f_i(z) = f_{\text{log}}(z) := \ln(1+\exp({-z}))$ for all $i=1,\dots,N$. The nonsmooth part is given by $\phi(x) = \lambda \|x\|_1$.

Since $f_{\text{log}}$ is convex, we can set $\gamma_i = 0$ for all $i$. The conjugate $f_{\text{log}}^\ast$ of the logistic loss function is given by \cite{Koh2007}:
\begin{align}
f_{\text{log}}^\ast(z) = 
\begin{cases}
-z\ln(-z) + (1+z)\ln(1+z) \quad &-1 < z< 0\\
+\infty \quad &\text{otherwise}.
\end{cases}
\end{align}
The mapping $f_{\text{log}}^\ast$ is $\mathcal{C}^\infty$ on $(-1,0)$ and locally Lipschitz. For all $z \in (-1,0)$, we have
$(f_{\text{log}}^\ast)^\prime(z) = \ln(1+z)-\ln(-z)$ and $(f_{\text{log}}^\ast)^{\prime\prime}(z) = - \frac{1}{z^2+z} \geq 4$.
We conclude that $f_{\text{log}}^\ast$ is essentially differentiable and strongly convex on $(-1,0)$.

The proximity operator of $\phi$ and its Clarke differential are discussed in, e.g., \cite{Zhang2020}. The proximity operator is the well-known soft-thresholding operator $\prox{\lambda \|\cdot\|_1}(x) = \text{sign}(x) \odot \max\{|x| - \lambda \mathds{1}, 0\}$ where ``$\odot$'' denotes component-wise multiplication. We can choose the generalized derivative $D \in \partial \prox{\lambda \|\cdot\|_1}(x)$ as follows: $D = \mathrm{Diag}(d_i)_{i=1,\dots,n}$ and $d_i = 0$ if $|x_i| \leq \lambda$ and $d_i = 1$ if $|x_i| > \lambda$. Lem.\ 2.1 in \cite{Zhang2020} ensures that $\prox{\lambda \|\cdot\|_1}$ is strongly semismooth. Altogether, \cref{asum:main} and \ref{asum:conjugate} are satisfied (due to strong semismoothness of $\prox{\lambda \|\cdot\|_1}$, \cref{asum:semismooth_alternative} holds as well).
\subsubsection{Description of Datasets}
We use several standard datasets for our experiments, listed in \cref{table1}.\footnote{Dataset \texttt{sido0} is downloaded from \url{http://www.causality.inf.ethz.ch/challenge.php?page=datasets\#cont}, \texttt{higgs} from \url{https://archive.ics.uci.edu/ml/datasets/HIGGS}, and \texttt{mnist} from \url{openml.org} using the \texttt{scikit-learn} API. All other datasets are downloaded from LIBSVM, \url{https://www.csie.ntu.edu.tw/~cjlin/libsvmtools/datasets/binary.html}.} 
The dataset \texttt{mnist} contains $28\times28$ pixel pictures of hand-written digits \cite{LeCun2010}. In order to obtain binary labels, we classify the two sets of digits $\{0,3,6,8,9\}$ and $\{1,2,4,5,7\}$.  \texttt{gisette} is derived from \texttt{mnist} but with additional higher-order and distractor features \cite{Guyon2005}. 
The \texttt{madelon.2} dataset is obtained as follows: we first scale the original \texttt{madelon} dataset (from LIBSVM) having 500 features and 2000 samples obtaining mean-zero and unit-variance features. Then, we apply a polynomial feature expansion of degree two, i.e., we add all pairwise products of features and a constant feature, resulting in $n=125751$. 
For \texttt{mnist} and \texttt{higgs}, we apply standard preprocessing to obtain mean-zero and unit-variance features. The other datasets are already suitably scaled and therefore not preprocessed.
For \texttt{mnist}, \texttt{gisette}, \texttt{sido0}, \texttt{covtype}, and \texttt{news20}, we use 80\% of the dataset samples for training and the remaining 20\% are used as validation set.
Note that \texttt{SNSPP} and \texttt{SVRG} compute the full gradient for the first time at the starting point $\tilde{x}^0$. In contrast to \texttt{SVRG}, the first iterate of \texttt{SNSPP} is not deterministic and therefore high variance in the gradient at the starting point could lead to unfavorable performance. In particular, we observe such effect for the \texttt{sido0} dataset. We find that this behavior can be easily prevented by running one iteration of \texttt{SNSPP} without variance reduction before computing the full gradient. However, for better comparability, for \texttt{sido0} we run one epoch of \texttt{SAGA} and use the final iterate as starting point for all methods. For all other datasets, we use $\tilde{x}^0=0$ as initial point for all algorithms.

\begin{table}[t]
	\centering
	\begin{tabular}{lcccc}  
		\cmidrule[1pt](){1-5} 
		Dataset & $N$ & $n$ & $\psi^\star$ & $\lambda$ \\[0.5ex]  
		\cmidrule(){1-5}  \\[-2.5ex]
		\texttt{mnist} & 56000 & 784 & 0.552 & 0.02 \\ 
		\texttt{gisette} & 4800 & 4955 & 0.476 & 0.05 \\ 
		\texttt{sido0} & 10142 & 4932 & 0.146 & 0.01 \\ 
		\texttt{covtype} & 581012 & 54 & 0.642 & 0.005 \\
		\texttt{higgs} & $11\cdot 10^6$ & 28 & 0.657 & 0.005 \\
		\texttt{madelon.2} & 2000 & 125751 & 0.568 & 0.02\\
		\texttt{news20} & 15996 & 1355191 & 0.653 & 0.001 \\
		\cmidrule[1pt](){1-5} \\[-1.5ex]
	\end{tabular}
	\caption{Information of the different datasets for sparse logistic regression. The optimal objective function value $\psi^\star$ is rounded to $3$ digits.}
	\label{table1}
\end{table}
\subsubsection{Subproblem Complexity}
We first illustrate the impact of solving the subproblems, i.e., invoking \cref{alg:semismooth_newton}, on the overall performance of \texttt{SNSPP}. \cref{fig:batch_size} depicts the subproblem complexity (in terms of runtime) and the overall progress of \texttt{SNSPP} for different choices of batch sizes using the \texttt{news20} dataset. As the batch size $b$ determines the dimension of the subproblem, we see a sharp increase in the runtime for larger choices of $b$ (bottom right). However, a larger batch size also allows to take larger steps and therefore more progress per iteration can be made as demonstrated in the left plot of \cref{fig:batch_size}. In our experiments, we typically observe that the subproblems can be solved very efficiently for batch sizes up to the order of few hundreds. For much larger batch sizes, the resulting higher computational costs of the subproblem will start to outweigh the benefits of reducing the variance of $\nabla f_\mathcal{S}$.
%
%
\subsubsection{Stability Analysis} 
The main focus of our numerical test is put on stability experiments with respect to hyperparameter selection, in particular, the step size. In practical scenarios, it is unlikely that a solver is executed with an intensively tuned step size due to tuning budgets. Hence, it is important to evaluate optimization methods considering the amount of step size tuning needed to reach optimal performance/runtime \cite{Schmidt2021}.
A similar comparison of \texttt{SPP} and \texttt{SGD} was conducted in \cite{Asi2019,Davis2019}, but without variance reduction, on a single batch, and with synthetic data only.
We compare \texttt{SNSPP} to the other variance-reduced methods \texttt{SAGA} and \texttt{SVRG}.  
We solve the $\ell_1$-regularized logistic regression problem for several datasets, for a range of step sizes $\alpha$ and different batch sizes $b$.\footnote{We always include results for \texttt{SAGA} with $b=1$ as this setting is widely adopted, for example in \texttt{scikit-learn}.}
For \texttt{SVRG}, we set the inner loop length to $\lfloor N/b \rfloor$.

\begin{figure}[t!]
	\centering
	\includegraphics[width=0.9\linewidth,trim=0ex 0ex 0ex 0ex, clip]{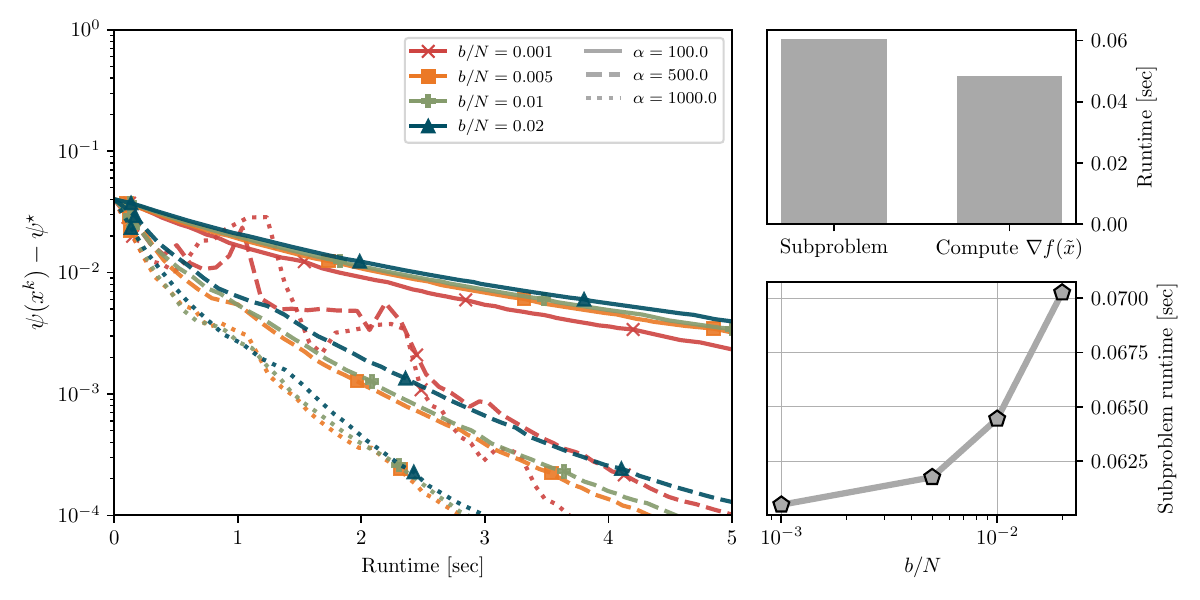}
	\caption{Sparse logistic regression for \texttt{news20}. Left: Objective function gap for different batch and step sizes. Top Right: average runtime for solving the subproblem once (with \cref{alg:semismooth_newton}) and for computing $\nabla f(\tilde x)$. Displayed for $\alpha=1000,b=0.005\cdot N$. Bottom Right: Mean runtime (per iteration index $k$ in \texttt{SNSPP}) for solving the subproblem.}
	\label{fig:batch_size}
\end{figure}

The tested algorithms terminate at iteration $k$, if the criterion 
\begin{align}
\label{eqn:stability_crit_convergence}
\psi(x^k)\leq 1.0001\psi^\star
\end{align} 
is satisfied. The runtime elapsed until fulfilling \eqref{eqn:stability_crit_convergence}, averaged over five independent runs, is plotted in \cref{fig:stability}.
The shaded area depicts the bandwidth of two standard deviations (over the five independent runs). If a method does not satisfy \eqref{eqn:stability_crit_convergence} within some maximum number of iterations or if it diverges, it is marked as \textit{no convergence}.

\paragraph{Discussion}
First, we observe that for all instances, \texttt{SNSPP} converges for much larger step sizes than \texttt{SAGA} and \texttt{SVRG}. The  elapsed runtime until convergence of \texttt{SNSPP} is robust to step size selection across all datasets. For \texttt{mnist}, \texttt{covtype} and \texttt{higgs}, the robustness of \texttt{SNSPP} is comparable to \texttt{SAGA} and slightly better than \texttt{SVRG}.
For \texttt{gisette}, \texttt{sido0}, and \texttt{madelon.2} (all of which are datasets where $n$ is large(r)) the advantage of \texttt{SNSPP} is most pronounced: for \texttt{madelon.2}, the runtimes of the best parameter settings are: \texttt{SNSPP}: 73 sec, \texttt{SAGA}: 132 sec, \texttt{SVRG}: 413 sec. \texttt{SNSPP} further converges in less than $300$ seconds for a large range of step sizes (i.e., without extensive tuning) while for \texttt{SAGA} and \texttt{SVRG} the runtime steeply increases beyond $500$ seconds if the step size is chosen too small (see \cref{fig:madelon} for a convergence plot).
Our results underpin the numerical evidence in \cite{Asi2019, Davis2019} that implicit stochastic proximal point methods tend to be more robust with respect to step size choices than stochastic gradient descent-type approaches.

\begin{figure}[h!]
	\begin{subfigure}[h!]{0.49\textwidth}
		\centering
		\includegraphics[width=\linewidth,trim=3ex 0ex 6ex 2.5ex, clip]{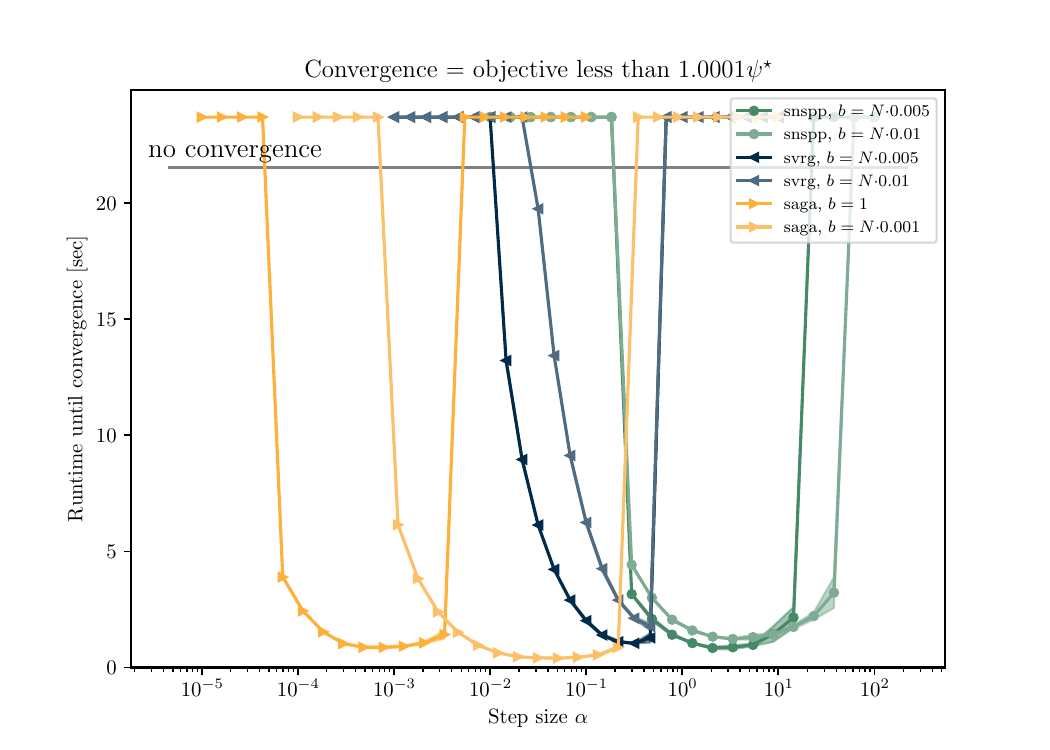}
		\caption{\texttt{mnist}}
	\end{subfigure}
	\begin{subfigure}[h!]{0.49\textwidth}
		\centering
		\includegraphics[width= \linewidth,trim=3ex 0ex 6ex 2.5ex, clip]{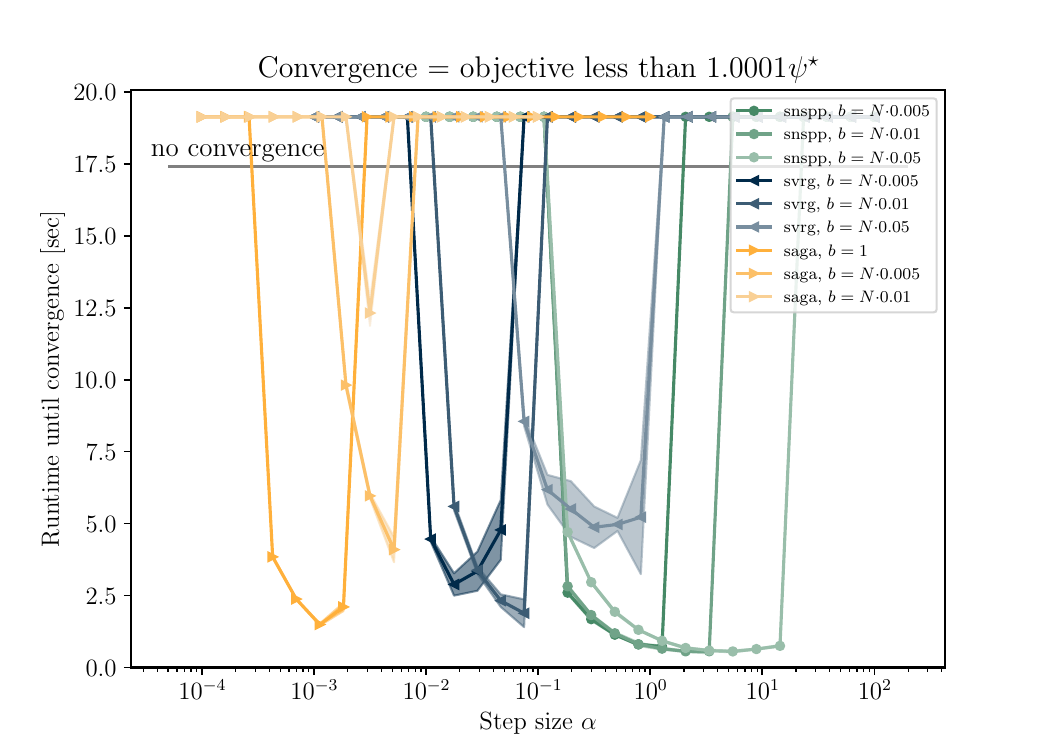}
		\caption{\texttt{gisette}}
	\end{subfigure}
	\begin{subfigure}[h!]{0.49\textwidth}
		\centering
		\includegraphics[width= \linewidth,trim=3ex 0ex 6ex 2.5ex, clip]{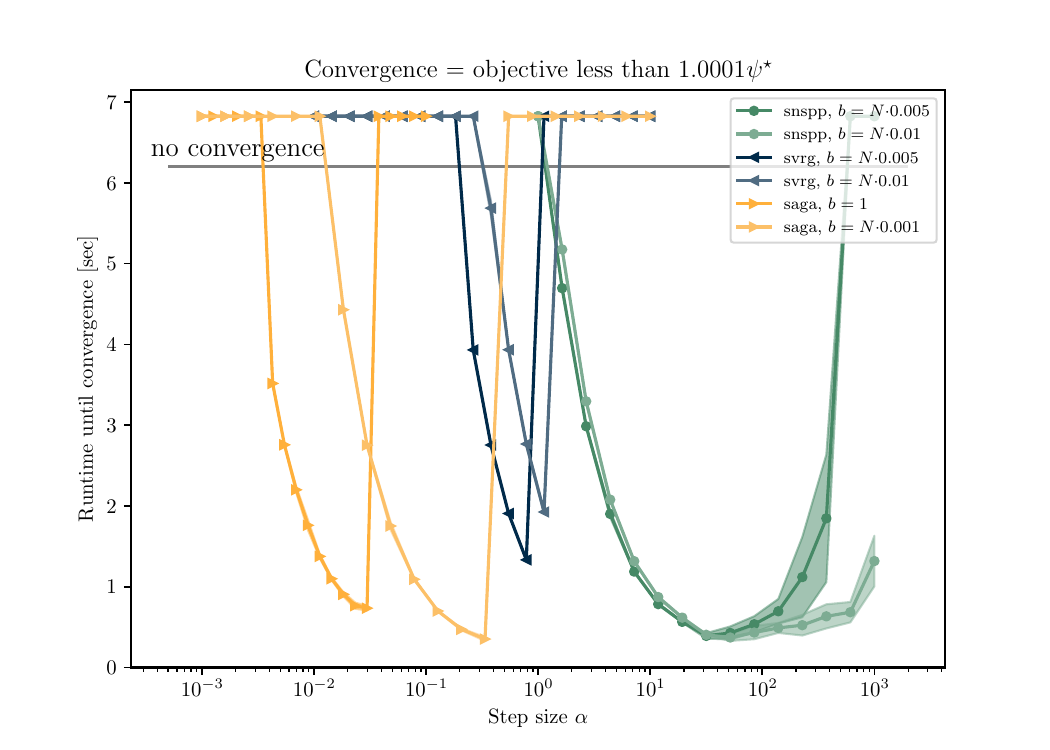}
		\caption{\texttt{sido0}}
	\end{subfigure}
	\begin{subfigure}[h!]{0.49\textwidth}
		\centering
		\includegraphics[width= \linewidth,trim=3ex 0ex 6ex 2.5ex, clip]{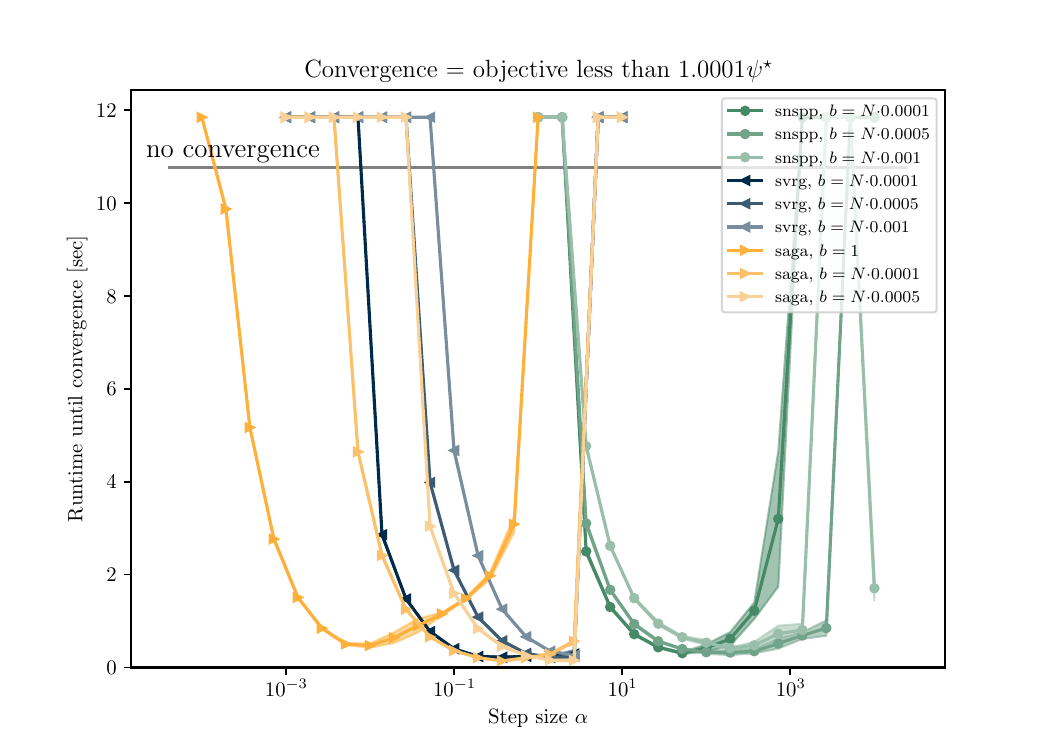}
		\caption{\texttt{covtype}}
	\end{subfigure}
	\begin{subfigure}[h!]{0.49\textwidth}
		\centering
		\includegraphics[width= \linewidth,trim=3ex 0ex 6ex 2.5ex, clip]{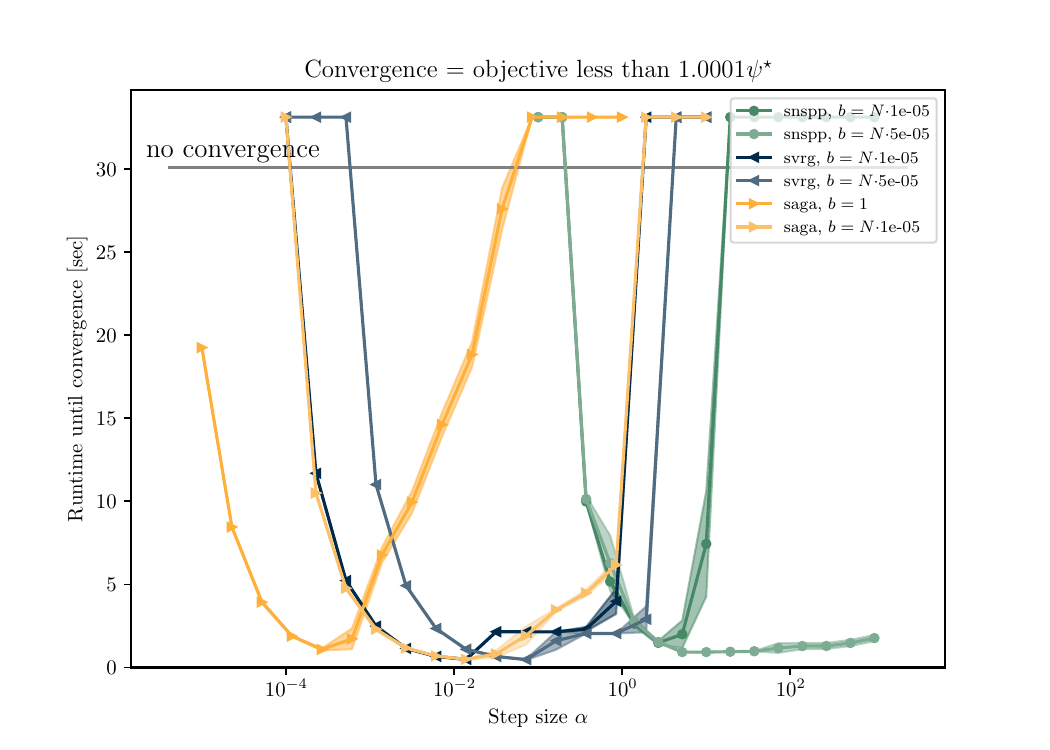}
		\caption{\texttt{higgs}}
	\end{subfigure}
	\begin{subfigure}[h!]{0.49\textwidth}
		\centering
		\includegraphics[width= \linewidth,trim=3ex 0ex 6ex 2.5ex, clip]{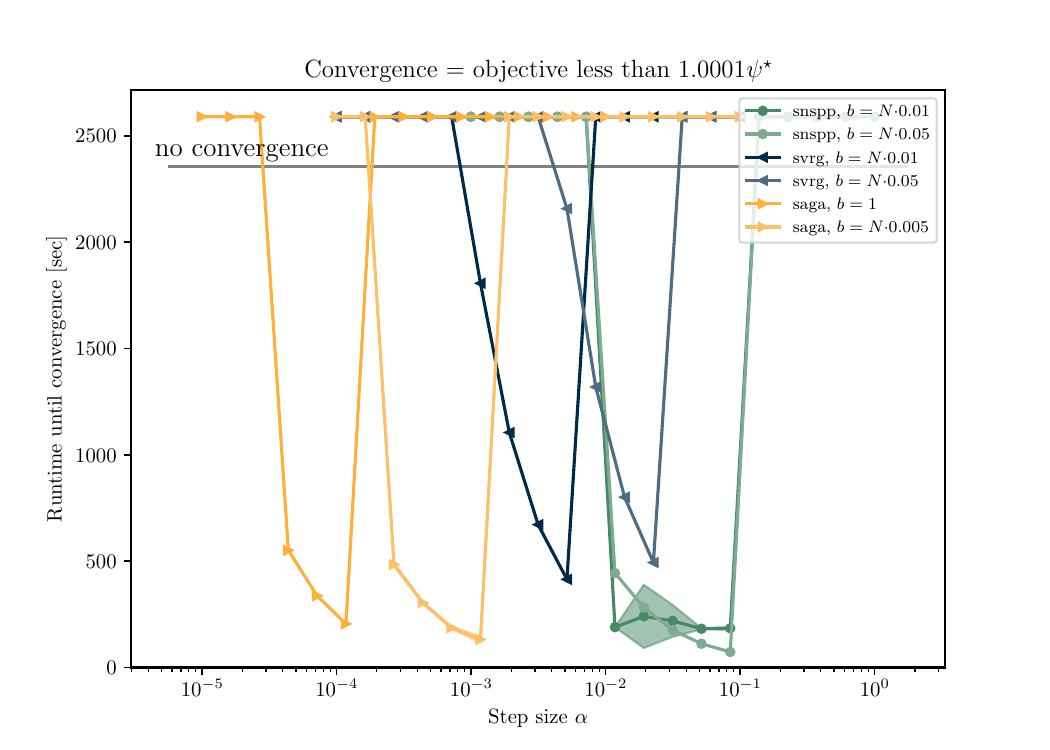}
		\caption{\texttt{madelon.2}}
	\end{subfigure}
	\caption{Runtime until convergence for different choices of step and batch sizes.}
	\label{fig:stability}
\end{figure}
\subsubsection{Speed of Convergence} \label{sec:logreg_comparison}
Based on the stability results depicted in \cref{fig:stability}, we now illustrate the speed of convergence of \texttt{SNSPP} compared to \texttt{SAGA}, \texttt{SVRG} and \texttt{AdaGrad}. For the experiments in this section, we choose a manually tuned (constant) batch and step size for all methods in order to allow a fair comparison. 
Details on the tuning procedure and the specific batch and step sizes values are reported in \cref{table2:hyperparams} in \cref{sec:table-hyperparams}.

We plot the objective function value -- averaged over 20 independent runs -- over the (average) cumulative runtime in \cref{fig:logreg_obj}. Due to the incorporated variance reduction, \texttt{SNSPP} converges to the optimal value and requires a relatively low number of iterations (compared to the other methods) in order to reach a high accuracy solution. However, in each iteration we need to run \cref{alg:semismooth_newton} instead of having a closed-form update. Overall in terms of runtime, for \texttt{mnist} and \texttt{covtype}, \texttt{SNSPP} is slightly slower than \texttt{SAGA/SVRG} but still competitive. For \texttt{sido0} and \texttt{gisette}, \texttt{SNSPP} is the fastest method. We also plot the convergence in terms of gradient evaluations, ignoring all other computational costs, in \cref{fig:logreg_obj2}.
\begin{figure*}[t]
\centering
\begin{subfigure}[h!]{0.475\textwidth}
\centering
\includegraphics[width=\textwidth]{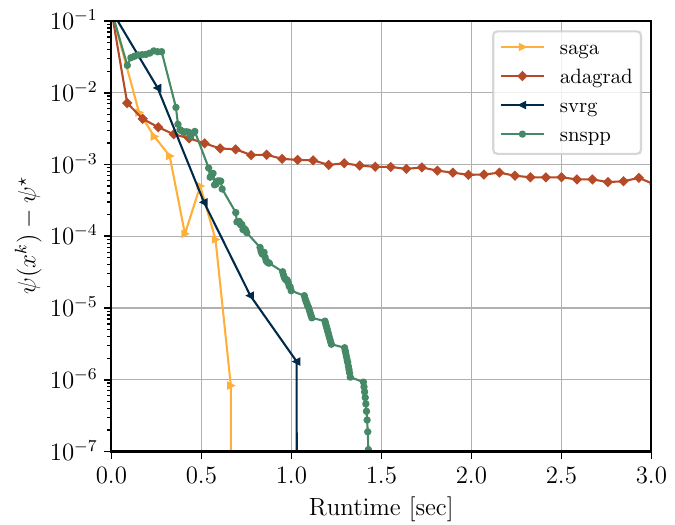}
\caption{\texttt{mnist}}
\label{fig:obj_mnist}
\end{subfigure}
\hfill
\begin{subfigure}[h!]{0.475\textwidth}  
\centering 
\includegraphics[width=\textwidth]{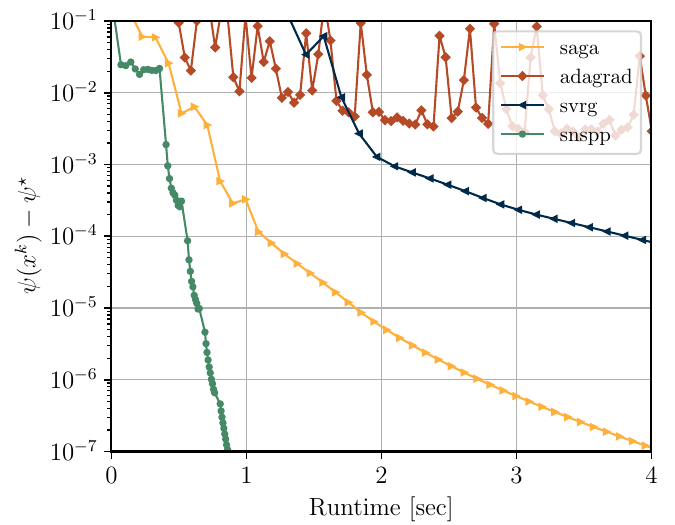}   
\caption{\texttt{gisette}}
\label{fig:obj_gisette}
\end{subfigure}
\begin{subfigure}[h!]{0.475\textwidth}   
\centering 
\includegraphics[width=\textwidth]{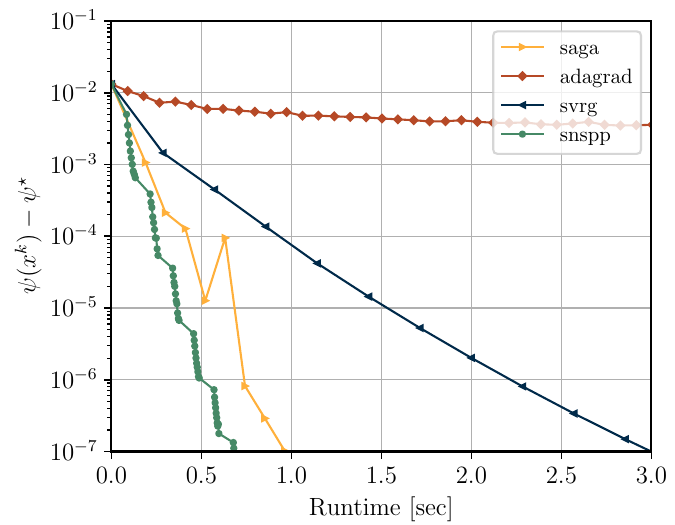}
\caption{\texttt{sido0}}
\label{fig:obj_sido1}
\end{subfigure}
\hfill
\begin{subfigure}[h!]{0.475\textwidth}   
\centering 
\includegraphics[width=\textwidth]{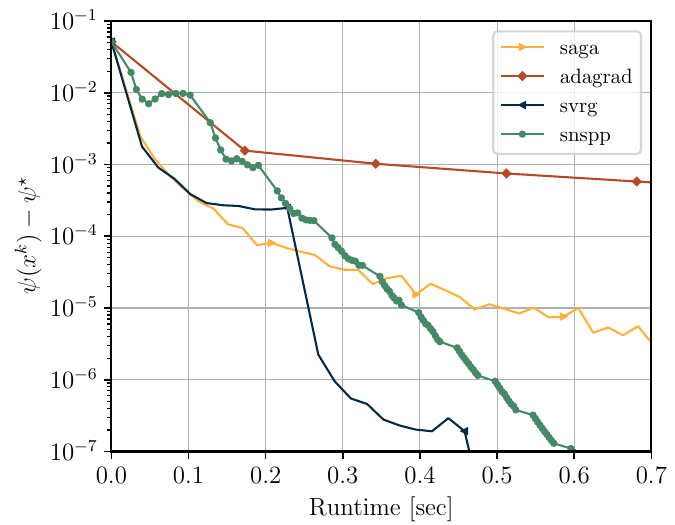}
\caption{\texttt{covtype}}
\label{fig:obj_covtype}
\end{subfigure}
\caption{Objective function convergence for the logistic regression datasets. 
For \texttt{SAGA} and \texttt{AdaGrad}, one marker denotes one epoch. For \texttt{SVRG} one marker denotes one outer-loop iteration while for \texttt{SNSPP} it denotes one (inner-loop) iteration.
}
\label{fig:logreg_obj}
\end{figure*}
%
\subsection{Sparse Student-t Regression}
\label{sec:experiments_student}
Next, for a given matrix $A \in \R^{N\times n}$ (with rows $a_i$) and measurements $b = (b_1,\dots, b_N) \in \R^N$, we consider sparse regression problems of the form  
\begin{align}
\label{prob:linear_regression}
\min_x~\mathcal{L}(Ax-b) + \lambda \|x\|_1.
\end{align}
where $\lambda >0$ is a regularization parameter and $\mathcal{L}:\R^m \to \mathbb{R}$ is a loss function.\\
In statistical learning, problem \cref{prob:linear_regression} with the squared loss $\mathcal{L}(r)=\oneover{N}\|r\|^2$ is known as the Lasso \cite{Tibshirani1996}. 
Other regularization terms have been proposed in order to model group sparsity or ordered features \cite{Simon2013}. 
While the squared loss is suitable when $b$ is contaminated by Gaussian noise, more heavy-tailed distributions have been studied in the presence of large outliers \cite{Huber1981}. For instance, using the Student-t distribution \cite{Aravkin2012,Lange1989} and the respective maximum-likelihood loss, problem \cref{prob:linear_regression} becomes
\begin{align}
\label{eqn:tstudent_problem}
\min_x~\oneover{N}\sum_{i=1}^{N} \ln \left(1+\hat\nu^{-1}{(\langle a_i,x\rangle-b_i)^2}\right) + \lambda \|x\|_1.
\end{align}
where $\hat\nu >0$ is the degrees of freedom-parameter of the Student-t distribution. Problem \cref{eqn:tstudent_problem} is of the form \cref{prob:deterministic} with $f_i(z) = \ln(1+\hat\nu^{-1}{(z-b_i)^2})$. 

We now fix $b \in\mathbb{R}$ and consider the scalar function $f_\text{std}: \R \to \R$, $f_\text{std}(x):= \ln(1+\hat\nu^{-1}{(x-b)^2})$; we have 
$f_\text{std}^\prime(x) = \frac{2(x-b)}{\hat\nu + (b - x)^{2}}$ and $f_\text{std}^{\prime\prime}(x) = \frac{2(\hat\nu - (b - x)^{2})}{(\hat\nu + (b - x)^{2})^{2}}$.
The minimum of $f_\text{std}^{\prime\prime}$ is attained at $x\in\{b+\sqrt{3\hat\nu}, b-\sqrt{3\hat\nu}\}$ and we can conclude $\inf_x\,f_\text{std}^{\prime\prime}(x) = -\frac{1}{4\hat\nu}$. Consequently, $f_\text{std}$ is $\frac{1}{4\hat\nu}$-weakly convex. 
Next, we compute the convex conjugate of $x\mapsto \hat f_{\text{std}}(x) := f_\text{std}(x) + \frac{\gamma}{2} x^2$ which is strongly convex for $\gamma > \frac{1}{4\hat\nu}$. For fixed $x \in \mathbb{R}$, it holds that
\begin{align}
\begin{split}
& \;z = \arg {\sup}_y~x y - \hat{f}_\text{std}(y) \iff \;x -f_\text{std}'(z) -\gamma z = 0 \\
\iff &\; -\gamma z^3+ z^2\left(x+2\gamma b\right) + z\left(-2bx-2-\gamma\hat\nu-\gamma b^2\right) +\left(x\hat\nu+xb^2+2b\right) = 0. \label{eqn:foc_fstar}
\end{split}
\end{align}
Choosing $\gamma > \frac{1}{4\hat\nu}$, \cref{eqn:foc_fstar} has a unique real solution $z^\ast$ for any $x,b \in \mathbb{R}$ due to strong convexity. Applying \cref{lem:conjugate_deriv} yields
\begin{align*}
\hat{f}_\text{std}^\ast(x) &= x z^\ast - \hat{f}_\text{std}(z^\ast), \quad (\hat{f}_\text{std}^\ast)^\prime(x) = z^\ast, \quad
(\hat{f}_\text{std}^\ast)^{\prime\prime}(x) = (\hat{f}_\text{std}^{\prime\prime}(z^\ast))^{-1}. 
\end{align*}
We solve the cubic polynomial equation in \cref{eqn:foc_fstar} using Halley's method \cite{Deiters2014}.
We run two different settings:
First, we use a synthetic dataset with $n=5000$, $N=4000$, $N_{\text{test}}=400$, $\lambda=0.001$, and $\hat{\nu}\in\{0.5,1,2\}$. We generate $\hat{x}\in\R^n$ with $20$ non-zero entries. To obtain $A$ and $b$, we first perform a SVD of a $(N+N_{\text{test}})\times n$-matrix with entries drawn uniformly at random from $[-1,1]$ and rescale its non-zero singular values to lie in the interval $[1,15]$. We use $\tilde A$ to denote the resulting larger matrix and we compute $\tilde b$ via $\tilde b = \tilde  A\hat{x} + 0.1\cdot\bar{\varepsilon}$ where $\bar{\varepsilon}\in \R^{N+N_{\text{test}}}$ is generated from a Student-t distribution with degrees of freedom $\hat\nu$. $A$ and $b$ are then given as the first $N$ rows/entries of $\tilde A$ and $\tilde b$. The remaining rows/entries are used as a test set.
Secondly, we consider problem \cref{eqn:tstudent_problem} using the feature matrix $A$ from the \texttt{sido0} dataset. We generate $\hat{x}$ with $50$ non-zero entries and compute $b=A\hat{x} + 0.1\cdot\bar{\varepsilon}$ where $\bar{\varepsilon}\in\R^N$ is generated from a Student-t distribution with degrees of freedom $\hat{\nu}=2$. As in the previous test, 20\% of the samples are used as test set (applying the same procedure) and we set $\lambda=0.01$.
We follow the same tuning strategy as described in \cref{sec:logreg_comparison}. The objective function and test loss are averaged over 20 independent runs. 

\paragraph{Discussion} For the synthetic data (\cref{fig:tstudent}), we observe that \texttt{SNSPP} performs comparably to \texttt{SVRG} and \texttt{SAGA} in reducing the objective as well as the Student-t likelihood loss on a held-out test set.
In order to exclude the possibility that the methods converge to different points with similar objective function values, we verified that the iterates of all methods follow a similar path and that the final iterates stay very close in terms of Euclidean distance. (Only the iterates generated by \texttt{AdaGrad} show a more oscillatory behavior which is expected as it does not use variance reduction).
\cref{fig:tstudent_sido} shows the results for the regression on \texttt{sido0}: here, \texttt{SNSPP} performs favorably compared to the other methods -- both in terms of objective function and test loss.   
\begin{figure}[t]
\centering
\begin{subfigure}[b]{0.32\textwidth}
\centering
\includegraphics[width=\linewidth]{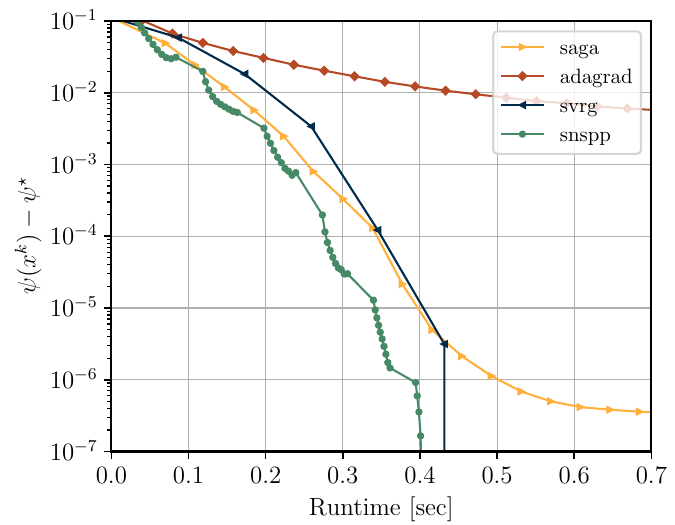}
\includegraphics[width= \linewidth]{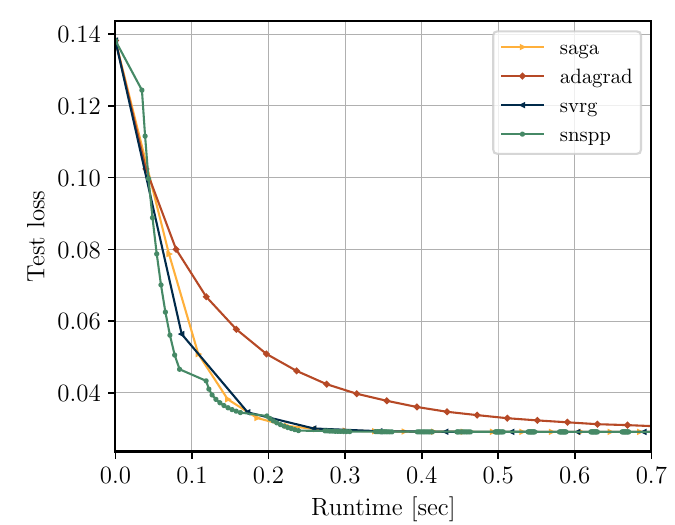}
\caption{$\hat{\nu} = 2$}
\label{fig:tstudent-1}
\end{subfigure}
\begin{subfigure}[b]{0.32\textwidth}
	\centering
	\includegraphics[width=\linewidth]{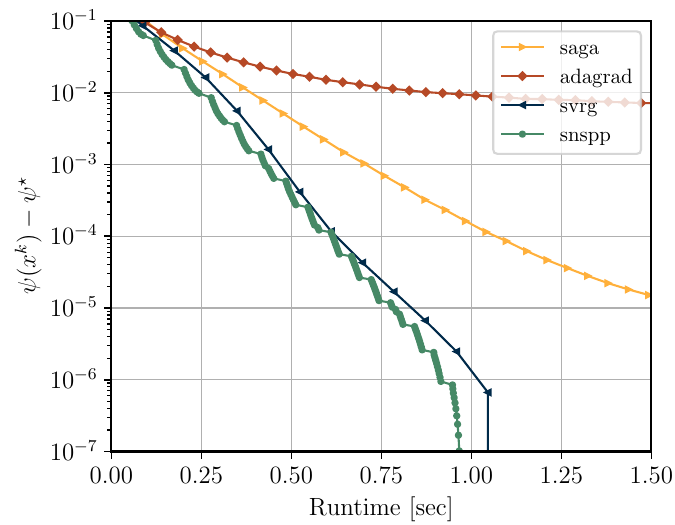}
	\includegraphics[width= \linewidth]{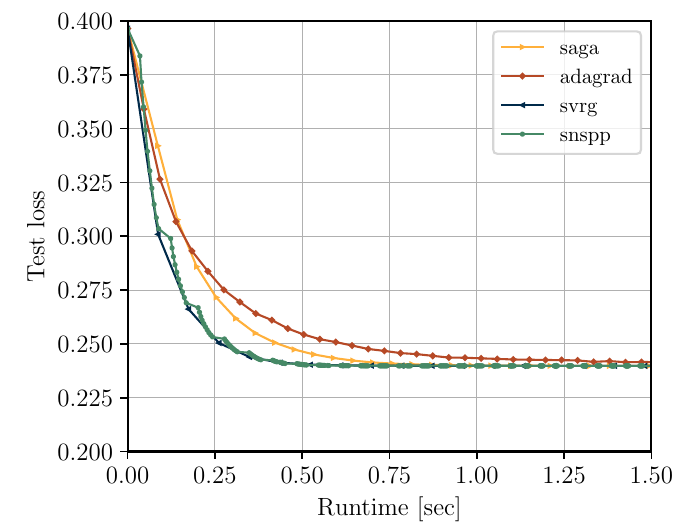}
	\caption{$\hat{\nu} = 1$}
	\label{fig:tstudent-2}
\end{subfigure}
\begin{subfigure}[b]{0.32\textwidth}
	\centering
	\includegraphics[width=\linewidth]{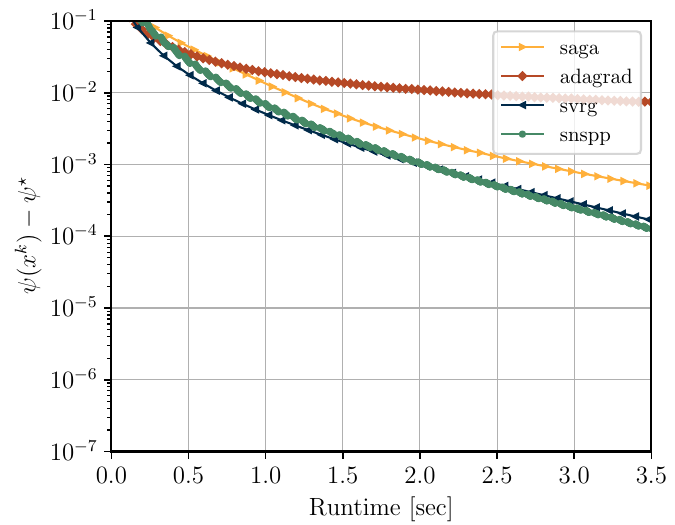}
	\includegraphics[width= \linewidth]{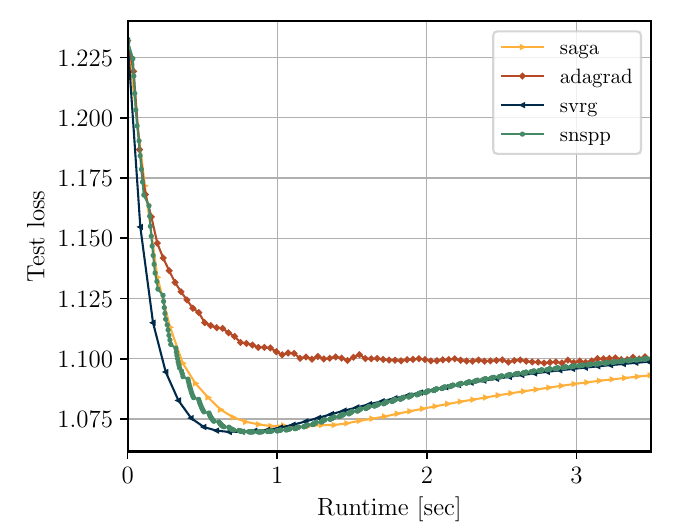}
	\caption{$\hat{\nu} = 0.5$}
	\label{fig:tstudent-4}
\end{subfigure}
\caption{Objective function (top) and test loss (bottom) for Student-t regression with different values of degrees of freedom. Test loss is defined as the value of $f_\text{std}$ averaged over the samples of the test set.}
\label{fig:tstudent}
\end{figure}
\begin{figure}[t]
	\centering
	\begin{subfigure}[h!]{0.475\textwidth}
		\centering
		\includegraphics[width=\textwidth]{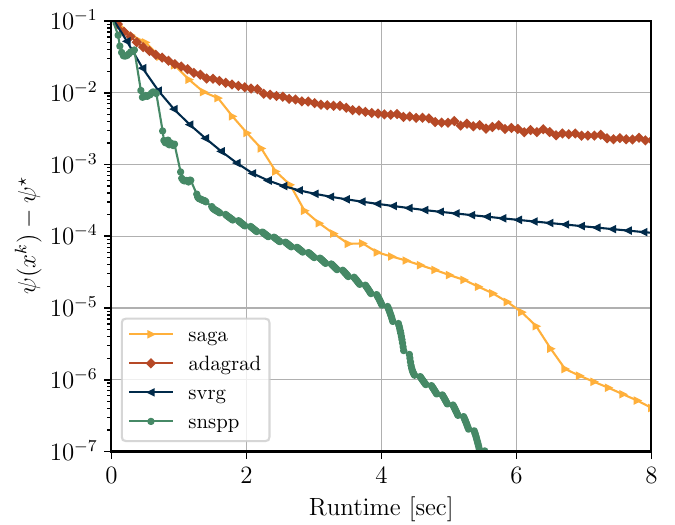}
		\caption{Objective function}
		\label{fig:tstudent_sido_obj}
	\end{subfigure}
	\hfill
	\begin{subfigure}[h!]{0.475\textwidth}  
		\centering 
		\includegraphics[width=\textwidth]{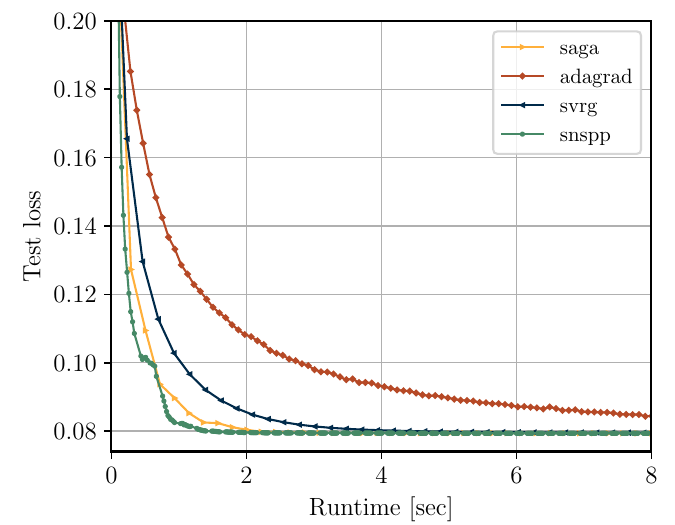}   
		\caption{Test loss}
		\label{fig:tstudent_sido_error}
	\end{subfigure}
\caption{Sparse Student-t regression for \texttt{sido0} dataset. Test loss is defined as the value of $f_\text{std}$ averaged over the samples of the test set.}
\label{fig:tstudent_sido}
\end{figure}
\section{Conclusion}
We develop a semismooth Newton stochastic proximal point method (\texttt{SNSPP}) for composite optimization that is based on the stochastic proximal point algorithm, the semismooth Newton method, and progressive variance reduction. The novel combination of stochastic techniques and of the semismooth Newton method to solve the occurring subproblems results in an effective stochastic proximal point scheme that is suitable for classes of weakly convex, nonsmooth, and large-scale problems. Convergence guarantees have been established that reflect similar theoretical results for \texttt{SVRG}. The proposed algorithm achieves promising numerical results and -- compared to other variance reduced gradient methods -- its performance is less sensitive with respect to tuning of the step size. This can be advantageous in practice when tuning budgets have to be considered.
\appendix
\section{Auxiliary Results}
\subsection{Preparatory Lemma}

\begin{lemma} \label{lem:conjugate_deriv}
Let $h:\R^n \to \R$ be a strongly convex, twice continuously differentiable mapping and let $z^\ast(x)$ denote the (unique) solution to $\max_z \langle x,z\rangle -h(z)$. Then, the convex conjugate $h^\ast: \R^n \to \R$ is $\mathcal{C}^2$ and for all $x\in \R^n$ it holds that
\begin{align*}
h^\ast(x) = \langle x, z^\ast(x)\rangle - h(z^\ast(x)), \quad \nabla h^\ast (x) = z^\ast(x), \quad \nabla^2 h^\ast (x) = \left[\nabla^2h(z^\ast(x))\right]^{-1}.
\end{align*}
\end{lemma}
\begin{proof}
As $h$ is strongly convex and $\mathcal{C}^2$, $z^\ast(x)$ is the unique solution to $\nabla h(z)=x$. By \cite[Thm.\ 4.20]{Beck2017}, we have $\nabla h(y) = x$ if, and only if, $\nabla h^\ast (x) = y$ for all $x,y \in \R^n$, which implies $\nabla h^\ast (x) = z^\ast(x)$. The inverse function theorem yields that $x\mapsto z^\ast(x)$ is $\mathcal{C}^1$ with Jacobian $Dz^\ast(x) = [\nabla^2 h(z^\ast(x))]^{-1}$; as $Dz^\ast(x)=\nabla^2 h^\ast(x)$ the statement is proven. 
\end{proof}

\subsection{Bounding the Variance}
In this section, let $\mathcal{F}$ be a $\sigma$-algebra and suppose that $x$ and $\tilde x$ are $\mathcal{F}$-measurable random variables in $\R^n$. For $i\in[N]$, let us further define $\zeta_i:= \trp{A_i} (\nabla f_{i}(A_i x) - \nabla f_{i}(A_i \tilde{x}))$. 
\begin{lemma}
\label{lem:bound_zeta_i}
Suppose that condition \ref{A1} is satisfied and let the index $i$ be drawn uniformly from $[N]$ and independently of $\mathcal{F}$. Conditioned on $\mathcal{F}$, we then have $\E\|\zeta_i\|^2 \leq {\mL}^2\|x-\tilde{x}\|^2$ almost surely. 
In addition, if every $f_i$ is convex and there exists $x^\star \in \argmin_{x}\psi(x)$, then it holds
\begin{align*}
\E\|\zeta_i\|^2 \leq 4{\mL}(\psi(x)-\psi(x^\star)+\psi(\tilde x)- \psi(x^\star)) \quad \text{almost surely}.
\end{align*}
\end{lemma}
\begin{proof}
The first statement follows directly from Lipschitz-smoothness. To prove the second part, let us define $\varphi_i: x \mapsto f_i(A_i x)$. Lem.\ 3.4 in \cite{Xiao2014} (applied to $\varphi_i$) implies
\begin{align*}
\nonumber \E\|\nabla \varphi_i(x)- \nabla \varphi_i(\tilde{x})\|^2 &\leq 2 \E\|\nabla \varphi_{i}(x) - \nabla \varphi_{i}(x^\star)\|^2 + 2 \E\|\nabla \varphi_{i}(\tilde{x}) - \nabla \varphi_{i}(x^\star)\|^2 \\
&\leq 4{\mL} (\psi(x) - \psi(x^\star) + \psi(\tilde{x})-\psi(x^\star) ), 
\end{align*}
where we used $\|a+b\|^2 \leq 2\|a\|^2 + 2\|b\|^2$. We conclude as $\zeta_i = \nabla \varphi_i(x)- \nabla \varphi_i(\tilde{x})$.
\end{proof}
\begin{lemma}
\label{lem:variance_replacement}
Let $\mathcal{S}$ be a $b$-tuple drawn uniformly from $[N]$ and independently of $\mathcal{F}$. Conditioned on $\mathcal{F}$, the random variable $u := \nabla f_{\mathcal S}(x) - \nabla f_{\mathcal S}(\tilde x) + \nabla f(\tilde x)$ is an unbiased estimator of $\nabla f(x)$. Moreover, almost surely, if $\mathcal{S}$ is drawn 
\begin{enumerate}[label=\textup{\textrm{(\roman*)}},topsep=0pt,itemsep=0ex,partopsep=0ex]
\item with replacement, then $\E\|u-\nabla f(x)\|^2 \leq \oneover{b}\E\|\zeta_i\|^2$ for any $i\in\mathcal{S}$.
\item without replacement, then $\E\|u-\nabla f(x)\|^2 \leq (1-\tfrac{b}{N})\oneover{b(N-1)}\sum_{i=1}^{N}\E\|\zeta_i\|^2$.
\end{enumerate}
\end{lemma}
\begin{proof}
Utilizing \cite[\S2.8]{Lohr2010}, we clearly have $\E[u]=\nabla f(x)$. 
For part (i), consider $\zeta = \oneover{b}\sum_{i \in \mathcal S}\zeta_i$.
It holds $\zeta = u - \nabla f(\tilde x)$ and, conditioned on $\mathcal{F}$, we obtain
\begin{align*}
\E\|u - \nabla f(x)\|^2  
= \frac{1}{b^2} \E \left\| {\sum}_{i \in \mathcal{S}} \zeta_i - \E[\zeta_i]\right\|^2
\end{align*}
since $\E[\zeta_i]=\nabla f(x) - \nabla f(\tilde{x})$ for all $i\in\mathcal{S}$. Since the random variables $\{\zeta_i - \E[\zeta_{i}]\}$ are i.i.d.\ and have
mean zero (conditioned on $ \mathcal{F}$), \cite[Lem.\ 7]{J.Reddi2016} allows to conclude
\begin{align}
\label{eqn:bound_variance_zeta}
\E\|u - \nabla f(x)\|^2  = \frac{1}{b^2}{\sum}_{i \in \mathcal{S}}\,\E \|\zeta_i - \E[\zeta_{i}]\|^2 \leq \frac{1}{b^2}{\sum}_{i \in \mathcal{S}}\,\E \|\zeta_i\|^2.
\end{align}
This proves the first statement as the random variables $\zeta_i$ are identically distributed.
The formula in (ii) is shown in \cite[\S2.8]{Lohr2010}.
\end{proof}
Combining \cref{lem:bound_zeta_i} and \cref{lem:variance_replacement}, we obtain the following result which extends Lem.\ 3 in \cite{J.Reddi2016} and Cor.\ 3.5 in \cite{Xiao2014}.
\begin{corollary} \label{cor:c1} Let \ref{A1} hold and let $\mathcal{S}$ be a $b$-tuple drawn uniformly from $[N]$ and independently of $\mathcal{F}$. With $u$ as in \cref{lem:variance_replacement} and conditioned on $\mathcal{F}$, it holds that
\begin{enumerate}[label=\textup{\textrm{(\roman*)}},topsep=0pt,itemsep=0ex,partopsep=0ex]
\item $\E\|u - \nabla f(x)\|^2 \leq  \frac{{\mL}^2\tau}{b}\|x-\tilde{x}\|^2$;
\item if all $f_i$ are convex and $x^\star \in \argmin_{x}\psi(x)$ exists, then
\[\E\|u - \nabla f(x)\|^2 \leq \tfrac{4{\mL}\tau}{b}(\psi(x)-\psi(x^\star)+\psi(\tilde x)- \psi(x^\star));\]
\end{enumerate} 
where $\tau = 1$ if $\mathcal{S}$ is drawn with replacement and $\tau=\frac{N-b}{N-1}$ if $\mathcal{S}$ is drawn without replacement.
\end{corollary}
%

\section{Proof for the Weakly Convex Case}\label{proof:thm:general_case}

\begin{proof}[Proof of \cref{thm:general_case}]
Let us fix the index of the outer loop $s$.
Recall the notation $\mathcal{A}_{\mathcal{S}_k}=\oneover{b}(\trp{A_{\ixmap_k(1)}},\dots,\trp{A_{\ixmap_k(b)}})^\top$
and abbreviate $\ixmap_k$ by $\ixmap$.
Let $(\hat{x}^{k+1},\hat\xi^{k+1})$ denote the pair of exact solutions of the implicit updates \cref{eqn:update_x_stochastic} and \cref{eqn:nonlinear_system_stochastic}.
In particular, setting $w^{k} := \mathcal{A}_{\mathcal S_k}^\top \xi^{k+1} - M_{\mathcal S_k}x^k  + v^k$ and $\hat w^{k} := \mathcal{A}_{\mathcal S_k}^\top \hat \xi^{k+1} - M_{\mathcal S_k}x^k  + v^k$, we have
\begin{equation} \label{eq:xhat-def} 
 {x}^{k+1}=\prox{\alpha_k \phi}(x^k- \alpha_kw^{k}) \quad \text{and} \quad \hat{x}^{k+1}=\prox{\alpha_k \phi}(x^k- \alpha_k\hat w^{k}). 
\end{equation}
We introduce the deterministic proximal update
\begin{equation} \label{eq:b-barx}
\bar x^{k+1} = \prox{\alpha_k \psi}^{I+\alpha_k M_N}(x^k) = \prox{\alpha_k\phi}(x^k-\alpha_k\nabla f(\bar x^{k+1}) - \alpha_k M_N(\bar x^{k+1}-x^k)).
\end{equation}
Using the Lipschitz smoothness of $f$, we obtain
\begin{equation}\label{est1}
f(x^{k+1}) \le f(x^k) + \langle\nabla f(x^k),x^{k+1}-x^k\rangle + \frac{L}{2} \|x^{k+1}-x^k\|^2.
\end{equation}
Setting $p = \prox{\gamma\phi}(y-\gamma w)$ and applying the optimality condition of the proximity operator \cref{eq:prox-char}, it holds that
\begin{equation} \label{eq:phi-diff} \phi(p) - \phi(z) \leq -\iprod{w}{p-z} + \frac{1}{2\gamma}\|y-z\|^2 - \frac{1}{2\gamma}\|p-y\|^2 - \frac{1}{2\gamma}\|p-z\|^2 \end{equation}
for all $y, w \in \Rn$, $z \in \dom(\phi)$, and $\gamma > 0$, see, e.g., \cite[Lem.\ 1]{J.Reddi2016} for comparison. 
Setting $y = x^k$, $\gamma = \alpha_k$, $w = w^{k}$, and $z = \bar x^{k+1}$, this yields 
\begin{equation}\label{est2}
\begin{aligned}
\phi(x^{k+1}) & \le \phi(\bar x^{k+1})
-\langle w^{k},x^{k+1}-\bar x^{k+1}\rangle
+\tfrac{1}{2\alpha_k}\|\bar x^{k+1}-x^k\|^2\\
&\hspace{4ex} -\tfrac{1}{2\alpha_k}\|x^{k+1}-x^k\|^2
-\tfrac{1}{2\alpha_k}\|x^{k+1}-\bar x^{k+1}\|^2.
\end{aligned}
\end{equation}
Moreover, setting $y = x^k$, $\gamma = \alpha_k$, $w = \nabla f(\bar x^{k+1})+M_N(\bar x^{k+1}-x^k)$, and $z = x^k$ in \cref{eq:phi-diff} and recalling \cref{eq:b-barx}, it follows
\begin{equation}\label{est3}
\phi(\bar x^{k+1})\le \phi(x^k)
-\langle \nabla f(\bar x^{k+1})+M_N(\bar x^{k+1}-x^k),\bar x^{k+1}-x^k\rangle
-\tfrac{1}{\alpha_k}\|\bar x^{k+1}-x^k\|^2.
\end{equation}
Adding the three inequalities \cref{est1}, \cref{est2}, and \cref{est3}, we obtain
\begin{align}
\nonumber \psi(x^{k+1}) & \leq \psi(x^k) + \iprod{\nabla f(x^k)-\nabla f(\bar x^{k+1})}{\bar x^{k+1}-x^k} + \iprod{\nabla f(x^k)-w^k}{x^{k+1}-\bar x^{k+1}} \\ & \hspace{2ex}- \|\bar x^{k+1}-x^k\|_{M_N+\frac{1}{2\alpha_k}I}^2 + \tfrac12\left[L - \tfrac{1}{\alpha_k}\right] \|x^{k+1}-x^k\|^2 - \tfrac{1}{2\alpha_k}\|x^{k+1}-\bar x^{k+1}\|^2.
\label{eqn:starting_point}
\end{align}
Let us define $\tilde v^k := \nabla f_{\mathcal{S}_k}(x^k)-\nabla f_{\mathcal{S}_k}(\tilde{x})+\nabla f(\tilde{x}) = \nabla f_{\mathcal{S}_k}(x^k) + v^k$. We decompose the term $\nabla f(x^k)-w^k$ as follows: 
\begin{align*}
\nabla f(x^k)-w^k & \\ & \hspace{-12ex}= \underbracket{\begin{minipage}[t][3.3ex][t]{21ex}\centering$\displaystyle- \mathcal{A}_{\mathcal S_k}^\top (\xi^{k+1} -\hat{\xi}^{k+1})$\end{minipage}}_{=: \, T_1} \, \underbracket{\begin{minipage}[t][3.2ex][t]{52ex}\centering$\displaystyle-M_{\mathcal S_k}(\hat x^{k+1}-x^{k+1}) - (\nabla f_{\mathcal S_k}(\hat x^{k+1})-\nabla f_{\mathcal S_k}(x^{k+1}))$\end{minipage}}_{=: \, T_2} \\ & \hspace{-8ex} \underbracket{\begin{minipage}[t][3.2ex][t]{47ex}\centering$\displaystyle- M_{\mathcal S_k}(x^{k+1}-x^{k}) - (\nabla f_{\mathcal S_k}(x^{k+1}) - \nabla f_{\mathcal S_k}(x^{k}))$\end{minipage}}_{=: \, T_3} \, + \,  \nabla f(x^k) - \tilde v^k,
\end{align*}
where we used $\mathcal{A}_{\mathcal S_k}^\top \hat\xi^{k+1} = \nabla f_{\mathcal S_k}(\hat x^{k+1})+M_{\mathcal S_k}\hat x^{k+1}$
(cf.\ \cref{eqn:update_xi_stochastic}, where $\xi^{k+1}$ and $x^{k+1}$ must
be replaced by $\hat\xi^{k+1}$ and $\hat x^{k+1}$, respectively, since in
\cref{alg:snspp} and beyond, $\xi^{k+1}$ and $x^{k+1}$ involve inexactness). Next, from the choice of $\epsilon_\text{sub}$ and \cref{prop:inexactness_bound}, we obtain $\|T_1\|^2 \leq \frac{\bar{A}^2}{\mu_*^2b}\epsilon_{k}^2$.
%
%
Applying Lipschitz smoothness, \cref{prop:inexactness_bound}, and using
$\|M_{\mathcal S_k}\| \leq \bar M$ (cf. \cref{eqn:definition_of_constants}), this yields 
\begin{align*}
\|T_2\| &\leq (\bar{L}_{b} + \|M_{\mathcal{S}_k}\|) \|\hat{x}^{k+1}-x^{k+1}\| \leq (\bar{L}_{b} + \bar M) \sqrt{\bar{A}^2 b^{-1}}\mu_*^{-1}\alpha_k \epsilon_{k} 
\end{align*}
for all $k =0,\dots,m-1$. Using $M_{\mathcal S_k} \succeq 0$ and Young's inequality, we have 
\begin{align*} \iprod{T_3}{x^{k+1}-\bar x^{k+1}} &= -\iprod{M_{\mathcal S_k}(\bar x^{k+1}-x^{k})}{x^{k+1}-\bar x^{k+1}} - \|x^{k+1}-\bar x^{k+1}\|_{M_{\mathcal S_k}}^2 \\ & \hspace{4ex} - \iprod{\nabla f_{\mathcal S_k}(x^{k+1}) - \nabla f_{\mathcal S_k}(x^{k})}{x^{k+1}-\bar x^{k+1}} \\ & \hspace{-8ex} \leq \left[\tfrac{\bar M}{2\sigma^3_k}+\tfrac{\bar{L}_{b}}{2\sigma^4_k}\right] \|x^{k+1}-\bar x^{k+1}\|^2 + \tfrac{\bar{L}_{b}\sigma^4_k}{2} \|x^{k+1}-x^k\|^2 + \tfrac{\bar M\sigma^3_k}{2} \|\bar x^{k+1}-x^k\|^2 \end{align*}
for some $\sigma^3_k, \sigma^4_k > 0$. Combining these results with \cref{eqn:starting_point}, applying Young's inequality, and defining $\nu^1_k := L+\tfrac{1}{2}\bar M\sigma^3_k - \tfrac{1}{2\alpha_k}$, $\nu^2_k := \tfrac12 [L+{\bar{L}_{b}}{\sigma^4_k} - \tfrac{1}{\alpha_k}]$, 
\[ \; \nu^3_k := \tfrac12 \left[\tfrac{1}{\sigma^1_k} + \tfrac{1}{\sigma^2_k} +\tfrac{\bar M}{\sigma^3_k}  + \tfrac{\bar{L}_{b}}{\sigma^4_k}+ \tfrac{1}{\sigma^5_k}- \tfrac{1}{\alpha_k}\right], \; \nu^4_k := \tfrac12\left[ \sigma^1_k + \sigma_k^2\alpha_k^2 (\bar{L}_{b}+\bar M)^2 \right] \tfrac{\bar{A}^2}{\mu_*^2b}, \]
we obtain
%
\begin{align*}
\psi(x^{k+1}) - \psi(x^k) & \\ & \hspace{-13ex} \leq\tfrac{\sigma^1_k}{2} \|T_1\|^2 + \tfrac{\sigma^2_k}{2} \|T_2\|^2 + \tfrac{\sigma^5_k}{2} \|\nabla f(x^k) - \tilde v^k \|^2 + \left[L+\tfrac{\bar M\sigma^3_k}{2} - \tfrac{1}{2\alpha_k}\right] \|\bar x^{k+1} - x^k\|^2  \\ & \hspace{-9ex} + \tfrac12 \left[L+{\bar{L}_{b}}{\sigma^4_k} - \tfrac{1}{\alpha_k}\right] \|x^{k+1} - x^k\|^2 + \nu^3_k \|\bar x^{k+1} - x^{k+1}\|^2 \\ & \hspace{-13ex} \leq \tfrac{\sigma^5_k}{2} \|\nabla f(x^k)-\tilde v^k\|^2 + \nu^1_k \|\bar x^{k+1} - x^k\|^2  + \nu^2_k \|x^{k+1} - x^k\|^2 \\ & \hspace{-9ex} + \nu^3_k \|\bar x^{k+1} - x^{k+1}\|^2 + \nu^4_k \epsilon_{k}^2. \end{align*}
Using \cref{cor:c1} (with $x=x^k, \tilde x = \tilde{x}^s$), conditioned on all random
events occurring until the current iterate $x^k=x^{s,k}$, we have
\[
\E\|\nabla f(x^k)-\tilde v^k\|^2
\le {\mL^2}{b^{-1}} \|x^k-\tilde{x}^s\|^2 .
\]
for all $k \in \{0,\dots,m-1\}$. Taking expectation in the previous estimate, it holds that
\begin{equation}
\label{eq:esti-prep}
\begin{aligned}
\E[\psi(x^{k+1})] &\le \E[\psi(x^k)] + \tfrac{\mL^2\sigma^5_k}{2b} \E\|x^k-\tilde{x}^s\|^2 + \nu^1_k \E\|\bar x^{k+1}-x^k\|^2 \\
&\hspace{4ex}+\nu^2_k \E\|x^{k+1}-x^{k}\|^2 + \nu^3_k \E\|\bar x^{k+1}-x^{k+1}\|^2 + \nu^4_k \epsilon_{k}^2. 
\end{aligned}
\end{equation}
With the goal of balancing the terms in $\nu_k^1,\dots,\nu_k^4$, we now choose $\sigma^1_k = m\alpha_k/(1-\bar\eta)$, $\sigma^2_k = {\sqrt{2b}}/[{\bar{L}_{b}\bar\eta}]$, $\sigma^3_k = \sigma^4_k = 1$, and $\sigma^5_k = {\sqrt{2b}}/{[\mL(m-1)]}$. 
For $k < m$, using $\tilde{x}^s=x^{0}$ and the Cauchy-Schwarz inequality, we deduce
\[
\E\|x^k-\tilde{x}^s\|^2 = \E\left\|{\sum}_{i=0}^{k-1} (x^{i+1}-x^i)\right\|^2
\le k \sum_{i=0}^{k-1} \E\|x^{i+1}-x^i\|^2
\le k \sum_{i=0}^{m-2} \E\|x^{i+1}-x^i\|^2.
\]
Summing this estimate for $k=0,\ldots,m-1$ gives
%
\begin{equation} \label{eq:vr-sumsum}
\sum_{k=0}^{m-1}\E\|x^k-\tilde{x}^s\|^2 \le
\sum_{k=0}^{m-1} k \sum_{i=0}^{m-2} \E\|x^{i+1}-x^i\|^2\le
\tfrac{m(m-1)}{2} \sum_{i=0}^{m-2} \E\|x^{i+1}-x^i\|^2.
\end{equation}
As $\sigma_k^5$ is independent of $k$, summing \cref{eq:esti-prep} over $k=0,\ldots,m-1$ gives
\begin{align}
\label{eq:ineq-nus}
\begin{split}
&\E[\psi(x^{m})]
\le \E[\psi(x^{0})] + \sum_{k=0}^{m-1} \left[ \tfrac{\mL^2m(m-1)\sigma^5_k}{4b} + \nu^2_k\right] \E\|x^{k+1}-x^k\|^2  \\
&\hspace{4ex} + \sum_{k=0}^{m-1} \nu^1_k \E\|\bar x^{k+1}-x^k\|^2 + \sum_{k=0}^{m-1}\nu^3_k\E\|x^{k+1}-\bar x^{k+1}\|^2 + \sum_{k=0}^{m-1}\nu^4_k\epsilon_{k}^2 .
\end{split}
\end{align}
Let $\bar\eta \in (0,1)$ be as stated in \cref{thm:general_case}. Utilizing the specific choices of $\sigma_k^1,\dots,\sigma_k^5$, we obtain ${\mL m}/{\sqrt{2b}} + 2\nu^2_k \leq -(1-\bar\eta)/\alpha_k$ if $\bar\eta{\alpha_k}^{-1} \geq [1+{m}/{\sqrt{2b}}] \mL + L$,  
\[ \nu^1_k \leq -\tfrac{1-\bar\eta}{2\alpha_k} \;  \iff \; \tfrac{\bar\eta}{\alpha_k} \geq 2L+\bar M, \quad \text{and} \quad \nu^3_k \leq 0 \; \Longleftarrow \; \tfrac{1}{\alpha_k} \geq \tfrac{(\bar M + \mL)m}{m-(1-\bar\eta)} + \tfrac{{\mL}m}{\sqrt{2b}}. \]
Thus, defining
\[
\tfrac{1}{\hat\alpha} := \tfrac{1}{\bar\eta}\max\left\{2L + \bar{M},\left[1+\tfrac{m}{\sqrt{2b}}\right] \mL +\max\{\bar M,L\} \right\},
\]
it holds $\nu_k^4 \leq \frac12[m/(1-\bar\eta)+\sqrt{2b}(1+\bar M/\bar{L}_{b})]\tfrac{\bar{A}^2}{\mu_*^2b} \alpha_k =: \bar\nu\alpha_k$. 

Next, introducing the auxiliary function $y \mapsto \psi_x(y) := \psi(y) + \frac12 \|y-x\|_{M_N}^2$, we can write $\bar x^{k+1} = \prox{\alpha_k\psi_{x^k}}(x^k)$. Consequently, applying \cite[Lem.\ 2]{Nesterov2013} and \cite[Thm.\ 3.5]{Dima2018} with $G = \partial\phi$, $\Phi = \partial\psi_{x^k}$, $t = \alpha_k$, and $\beta = L+\bar M$, it follows:
\begin{align} 
\label{eqn:prox_to_fnat}
\|\bar x^{k+1}-x^k\| \geq (1-(L+\bar M)\alpha_k) \|F_{\mathrm{nat}}^{\alpha_k}(x^k)\| \geq (1-\bar\eta)\alpha_k \|F_{\mathrm{nat}}(x^k)\|
\end{align}
as long as $\alpha_k \leq \min\{\hat\alpha,1\}$. We now use full notations showing the $(s,k)$-dependence. Let us define $\bar{x}^{s,k+1} := \prox{\alpha_k^s \psi}^{I+\alpha_k^sM_N}(x^{s,k})$,
\[\tau^1_s := {\sum}_{k=0}^{m-1}(\alpha_k^s)^{-1}\E\|\bar x^{s,k+1}-x^{s,k}\|^2, ~ 
\tau^2_s := {\sum}_{k=0}^{m-1}(\alpha_k^s)^{-1}\E\|x^{s,k+1}-x^{s,k}\|^2.
\] 
Using $\alpha_k^s\le\hat\alpha$, \cref{eq:ineq-nus} then implies
\begin{align}
\label{eqn:for_complexity_cor}
\E[\psi(\tilde{x}^{s+1})]-\E[\psi(\tilde{x}^s)] &\le
- \frac{1-\bar\eta}{2} (\tau^1_s+\tau^2_s) + \bar\nu \cdot {\sum}_{k=0}^{m-1}\alpha_k^s(\epsilon_k^s)^2 
\end{align}
and due to \ref{A2} and $\sum_{s=0}^\infty \sum_{k=0}^{m-1}\alpha_k^s(\epsilon_k^s)^2 < \infty$, we further get
%
%
\begin{equation} \label{eq:sum-result} {\sum}_{s=0}^\infty\tau^1_s < \infty \quad \text{and} \quad {\sum}_{s=0}^\infty \tau^2_s < \infty. \end{equation}
Finally, the estimates \cref{eqn:prox_to_fnat} and \cref{eq:sum-result} yield $\sum_{s=0}^\infty \sum_{k=0}^{m-1} \alpha_k^s \E\|F_{\mathrm{nat}}(x^{s,k})\|^2 < \infty$ and $\sum_{s=0}^\infty \sum_{k=0}^{m-1} \alpha_k^s \|F_{\mathrm{nat}}(x^{s,k})\|^2 < \infty$ almost surely. Due to $\sum_{s=0}^\infty\sum_{k=0}^{m-1} \alpha_k^s = \infty$ and the Borel-Cantelli lemma \cite{Durrett2019}, 
it holds $\liminf_{s \to \infty} \|F_{\mathrm{nat}}(x^{s,k})\|^2 = 0$ almost surely for all $0\leq k <m$. The almost sure convergence $F_{\mathrm{nat}}(x^{s,k}) \to 0$ (and $\E\|F_{\mathrm{nat}}(x^{s,k})\|\to 0$) essentially follows from \cref{eq:sum-result} and from the Lipschitz continuity of $x \mapsto F_{\mathrm{nat}}(x)$ and can be shown as in \cite[Thm.\ 3.3 and Thm.\ 4.1]{Yang2021}. As this last step is basically identical to the proofs in \cite{Yang2021}, we omit further details and refer to \cite{Yang2021}.
%
%
\end{proof}
\begin{proof}[Proof of  \cref{cor:general_case_rate}]
We have $\psi(\tilde{x}^{S+1}) \geq \psi^\star$ for all $S$. In the previous proof, \cref{eqn:for_complexity_cor} can be obtained directly from \cref{eq:ineq-nus}, using only the condition $\alpha_k^s \leq \hat{\alpha}$ on the step sizes (in particular, we do not need to assume $\alpha_k^s \leq \min\{\hat{\alpha},1\}$). Summing \cref{eqn:for_complexity_cor} from $s=0,\dots,S$, and setting $\bar x_\pi := \prox{\alpha \psi}^{I+\alpha M_N}(x_\pi)$, we conclude
\begin{align*}
\E\|\bar x_\pi-x_\pi\|^2
\le \tfrac{2\alpha} {(1-\bar{\eta})m(S+1)}\left[\psi(\tilde x^0)-\psi^\star+ \bar\nu\alpha\cdot {\sum}_{s=0}^S{\sum}_{k=0}^{m-1} (\epsilon_k^s)^2\right].
\end{align*}
As in \cref{eqn:prox_to_fnat}, we can now utilize the bound $\|\bar x_\pi-x_\pi\|  \geq (1-(L+\bar{M})\alpha) \|F_\text{nat}^\alpha(x_\pi)\| \geq (1-\bar{\eta}) \|F^\alpha_\text{nat}(x_\pi)\|$ to express complexity in terms of $\E\|F_\text{nat}(x_\pi)\|^2$.
\end{proof}
%
%
\section{Proof for the Strongly Convex Case}
\label{proofs:str_convex}
The proofs of \cref{thm:str_convex_case_obj} and \cref{thm:str_convex_case} have several identical steps. We start by proving the latter.
\begin{proof}[Proof of \cref{thm:str_convex_case}]
%
%
%
%
Fix $s \in \N_0$ and let $k\in\{0,\dots,m-1\}$ be given. Let again $(\hat{x}^{k+1},\hat\xi^{k+1})$ denote the pair of exact solutions of \cref{eqn:update_x_stochastic} and \cref{eqn:nonlinear_system_stochastic}.
Due to %
\begin{align}
\label{eqn:xi_hat_prop}
\hat{\xi}_i^{k+1} = \nabla f_{\ixmap(i)}(A_{\ixmap(i)} \hat{x}^{k+1}) + \gamma_{\ixmap(i)} A_{\ixmap(i)} \hat{x}^{k+1}, \quad i \in [b],
\end{align}
(cf.\ \cref{eqn:update_xi_stochastic} replacing again $(\xi^{k+1},x^{k+1})$ by $(\hat\xi^{k+1},\hat x^{k+1})$) and \cref{eq:xhat-def}, we have
\begin{align*}
\hat{x}^{k+1}  = \prox{\alpha\phi}(x^k - \alpha [\nabla f_{\mathcal S_k}(\hat x^{k+1})+v^k] -\alpha M_{\mathcal S_k}(\hat x^{k+1}-x^k)).
\end{align*}
Furthermore, introducing $\psi_k(x) := \psi_{\mathcal S_k}(x) + \iprod{v^k}{x-x^k}$, the underlying optimality condition of the proximity operator \cref{eq:prox-char} implies
\begin{align*} p = \hat x^{k+1} & \quad \iff \quad p \in x^k - \alpha M_{\mathcal S_k}(p-x^k) - \alpha[\partial \phi(p) + \nabla f_{\mathcal S_k}(p) + v^k] \\ & \quad \iff \quad p = \prox{\alpha\psi_{k}}^{I+\alpha M_{\mathcal S_k}}(x^k). \end{align*}
Moreover, using \ref{A1}, the mapping 
\[ x \mapsto \hat F_k(x) := f_{\mathcal S_k}(x) + \frac12\|x-x^k\|^2_{M_{\mathcal S_k}} = \oneover{b}{\sum}_{i \in \mathcal S_k} f_i(A_ix) + \frac{\gamma_i}{2} \|A_i(x-x^k)\|^2 \]
is convex for all $k$. Hence, setting $\Psi_k(x) := \hat F_k(x) + \phi(x) + \iprod{v^k}{x-x^k}$, the function $x \mapsto \Psi_k(x) + \frac{1}{2\alpha}\|x-x^k\|^2$ is $(\mu_\phi+{\alpha}^{-1})$-strongly convex and due to $\hat x^{k+1} = \argmin_{x} \Psi_k(x) + \frac{1}{2\alpha}\|x-x^k\|^2$, it follows
\begin{align} \label{eq:phi-k-esti} \Psi_k(x) + \tfrac{1}{2\alpha} \|x - x^k\|^2  \geq~& \Psi_k(\hat x^{k+1}) + \tfrac{1}{2\alpha} \|\hat x^{k+1} - x^k\|^2 \\
&+ \tfrac{1}{2}\left[\mu_\phi+\tfrac{1}{\alpha} \right] \|x - \hat x^{k+1}\|^2  \nonumber \end{align}
for all $x \in \dom(\phi)$. Next, combining the optimality condition \cref{eq:prox-char}, the update rule of \cref{alg:snspp}, and \cref{eqn:xi_hat_prop}, we obtain
\begin{align*}
x^{k+1} &\in x^k -\alpha\trp{\mathcal{A}}_{\mathcal S_k} \xi^{k+1} - \alpha v^k + \alpha M_{\mathcal S_k} x_k - \alpha \partial \phi(x^{k+1}) \\
&= x^k -\alpha\trp{\mathcal{A}}_{\mathcal S_k}(\xi^{k+1} - \hat \xi^{k+1}) - \alpha [\nabla f_{\mathcal S_k}(\hat x^{k+1})-\nabla f_{\mathcal S_k}(x^{k+1})] \\ & \hspace{-4ex}- \alpha M_{\mathcal S_k}(\hat x^{k+1}-x^{k+1}) - \alpha [\nabla f_{\mathcal{S}_k}({x}^{k+1}) + \partial \phi(x^{k+1}) + v^k + M_{\mathcal S_k}(x^{k+1}-x^{k})].
\end{align*}
Setting $h^{k+1} :=\trp{\mathcal{A}}_{\mathcal S_k}(\xi^{k+1} - \hat \xi^{k+1})+[\nabla f_{\mathcal S_k}(\hat x^{k+1})-\nabla f_{\mathcal S_k}(x^{k+1})] + M_{\mathcal S_k}(\hat x^{k+1}-x^{k+1})$, this shows that $-h^{k+1}+(x^k-x^{k+1})/\alpha \in \partial \Psi_k(x^{k+1})$. Thus, due to the strong convexity of $\Psi_k$ and applying $-\iprod{a}{b} = \frac12 \|a-b\|^2 - \frac12 \|a\|^2 - \frac12 \|b\|^2$, it follows
\begin{align*} \Psi_k(x^{k+1}) - \Psi_k(y) & \\ & \hspace{-15ex} \leq - \frac{1}{\alpha} \iprod{x^k-x^{k+1}}{y-x^{k+1}} + \iprod{h^{k+1}}{y-x^{k+1}} - \frac{\mu_\phi}{2}\|y-x^{k+1}\|^2 \\ & \hspace{-15ex} = \frac{1}{2\alpha}[ \|x^k - y\|^2 - \|x^{k+1}-x^k\|^2]  - \frac12\left[ \mu_\phi+\frac{1}{\alpha}\right] \|x^{k+1}-y\|^2 + \iprod{h^{k+1}}{y-x^{k+1}} \end{align*}
 for all $y \in \dom(\phi)$. Using this estimate in \cref{eq:phi-k-esti} with $y = \hat x^{k+1}$ and applying Young's inequality, we have
 \begin{align*} 
 \frac{1}{2\alpha} [(1+\alpha\mu_\phi)\|\hat x^{k+1} - x\|^2 - \|x^{k} - x\|^2] & \\ & \hspace{-34ex} \leq \Psi_k(x) - \Psi_k(x^{k+1}) + \Psi_k(x^{k+1}) - \Psi_k(\hat x^{k+1}) - \frac{1}{2\alpha} \|\hat x^{k+1}-x^k\|^2 \\ & \hspace{-34ex} \leq \Psi_k(x) - \Psi_k(x^{k+1}) - \frac{1}{2\alpha} \|x^{k+1}-x^k\|^2 - \frac{\mu_\phi}{2} \|\hat x^{k+1}-x^{k+1}\|^2  + \frac{\alpha}{2} \|h^{k+1}\|^2. \end{align*}
 %
Next, we expand the first term on the right hand side as follows:
 \begin{align*} \Psi_k(x) - \Psi_k(x^{k+1}) & = [\psi(x) - \psi(x^{k+1})] + [f_{\mathcal S_k}(x) - f(x)] + [\hat F_k(x^k) - \hat F_k(x^{k+1})] \\ & \hspace{-12ex} +\iprod{v^k}{x-x^{k+1}} + [f(x^{k+1})-f(x^k)]  + [f(x^k)-f_{\mathcal S_k}(x^k)] + \frac12 \|x-x^k\|_{M_{\mathcal S_k}}^2. \end{align*}
By the convexity of $\hat F_k$, we have $\hat F_k(x^k) - \hat F_k(x^{k+1}) \leq \iprod{\nabla \hat F_k(x^k)}{x^k-x^{k+1}} = \iprod{\nabla f_{\mathcal S_k}(x^k)}{x^k-x^{k+1}}$. Combining this with the Lipschitz continuity of $\nabla f$, it further holds that
%
\begin{align*} 
[\hat F_k(x^k) - \hat F_k(x^{k+1})] + [f(x^{k+1})-f(x^k)] & \\ & \hspace{-24ex} \leq \iprod{\nabla f(x^k)-\nabla f_{\mathcal S_k}(x^k)}{x^{k+1}-x^k} + \frac{L}{2} \|x^{k+1}-x^k\|^2. 
\end{align*}
%
In addition, using Young's inequality, we have
\[ \|\hat x^{k+1}-x\|^2 \geq (1- \rho_1) \|x^{k+1}-x\|^2 + \left[1 - \frac{1}{\rho_1}\right] \|\hat x^{k+1}-x^{k+1}\|^2 \]
for $\rho_1 \in (0,1)$. Together, applying Young's inequality once more and setting $e_k(x) := f_{\mathcal S_k}(x) - f(x) + f(x^k)-f_{\mathcal S_k}(x^k) + \iprod{v^k}{x-x^k}-\frac12 \|x^k-x\|^2_{M_N-M_{\mathcal S_k}}$, this yields
%
 \begin{align}
 \label{eqn:strconvex_obj_start}
 \begin{split}
 &(1+\alpha\mu_\phi)(1-\rho_1)\|x^{k+1} - x\|^2 - \|x^{k} - x\|^2  \\ & \hspace{4ex} \leq 2\alpha[\psi(x) - \psi(x^{k+1})] + 2\alpha e_k(x) + \alpha \|x-x^k\|_{M_N}^2 \\ & \hspace{4ex} + \frac{\alpha}{\rho_2} \|\nabla f(x^k)-\nabla f_{\mathcal S_k}(x^k)-v^k\|^2 - \left[1 - (L+\rho_2)\alpha \right] \|x^{k+1}-x^k\|^2 \\ & \hspace{4ex} + {\alpha^2} \|h^{k+1}\|^2 + (1+\alpha\mu_\phi)(\rho_1^{-1}-1) \|\hat x^{k+1}-x^{k+1}\|^2
 \end{split} 
 \end{align}
 %
for all $x \in \dom(\phi)$ and $\rho_2 > 0$, where we used $-\alpha\mu_\phi -(1+\alpha\mu_\phi)(1-\rho_1^{-1})\leq (1+\alpha\mu_\phi)(\rho_1^{-1}-1)$. The choice of $\epsilon_\text{sub}$ and \cref{prop:inexactness_bound} imply 
%
 $\|\trp{\mathcal{A}_{\mathcal S_k}}(\xi^{k+1}-\hat \xi^{k+1})\| \leq \tfrac{\bar{A}}{\mu_*\sqrt{b}} \epsilon_{k}$
%
and $\|\hat{x}^{k+1}-x^{k+1}\| \leq \tfrac{\alpha\bar{A}}{\mu_*\sqrt{b}} \epsilon_{k}$.
Moreover, applying Lipschitz smoothness and \cref{prop:inexactness_bound}, it holds that 
\begin{align*}
\|h^{k+1}\| &\leq \tfrac{\bar{A} }{\mu_*\sqrt{b}}\epsilon_{k} + (\bar{L}_{b} + \|M_{\mathcal{S}_k}\|) \|\hat{x}^{k+1}-x^{k+1}\| \leq (1+\alpha[\mL + \bar M]) \tfrac{\bar{A}}{\mu_*\sqrt{b}} \epsilon_{k}.
\end{align*}
We now choose $x = x^\star$; this yields
$\psi(x^\star) - \psi(x^{k+1}) \leq - \frac{\mu}{2} \|x^{k+1}-x^\star\|^2. $
Furthermore, \cref{cor:c1} (with $x=x^k, \tilde x = \tilde{x}^s$) yields $\E\|\nabla f(x^k)-\nabla f_{\mathcal S_k}(x^k)-v^k\|^2 \leq {\mL^2}{b^{-1}} \|x^k-\tilde x^s\|^2$ and we have $\E[e_k(x^\star)] = 0$. In addition, by definition and due to the Lipschitz continuity of $F_{\mathrm{nat}}$, we obtain $\epsilon_k \leq \delta_s \|F_{\mathrm{nat}}(\tilde x^s)\| \leq (2+L)\delta_s \|\tilde x^s - x^\star\|$. Using $M_N \preceq (\mu_\phi-\mu)I$, combining our previous results, and taking expectation, it follows
%
%
%
\begin{align*}
[1+\alpha(\mu_\phi+\mu)-\rho_1(1+\alpha\mu_\phi)] \E\|x^{k+1}-x^\star\|^2 & \\ & \hspace{-45ex} \leq [1+\alpha(\mu_\phi-\mu)] \E\|x^k-x^\star\|^2 + \frac{\mL^2\alpha}{b\rho_2} \E\| x^k-\tilde x^s\|^2  - [1-(L+\rho_2)\alpha] \E\|x^{k+1}-x^k\|^2 \\ & \hspace{-41ex} + \underbracket{\begin{minipage}[t][3.8ex][t]{52ex}\centering$\displaystyle\left[(1+\alpha\mu_\phi)\rho_1^{-1}+(1+\alpha[\mL+\bar M])^2\right] \tfrac{\bar{A}^2}{\mu_*^2b} (2+L)^2 \alpha^2$\end{minipage}}_{=: c(\alpha)} \delta_s^2 \E\|\tilde x^s - x^\star\|^2.
\end{align*}
We now suppose that $\rho_1$ is chosen such that $(1+\alpha\mu_\phi)\rho_1 \leq 2\alpha\mu$. Then, summing the last estimate for $k = 0, \dots,m-1$ and invoking \cref{eq:vr-sumsum}, this implies
\begin{align*} [1+\alpha(\mu_\phi+\mu)-\rho_1(1+\alpha\mu_\phi)] \E\|\tilde x^{s+1}-x^\star\|^2 & \\ & \hspace{-40ex} \leq [1+\alpha(\mu_\phi-\mu)+c(\alpha)m\delta_s^2] \E\|\tilde x^s-x^\star\|^2 \\ & \hspace{-36ex} - \left[1 - \left(L+\rho_2+\tfrac{\mL^2m(m-1)}{2b\rho_2}\right)\alpha\right] {\sum}_{k=0}^{m-1} \E\|x^{k+1}-x^k\|^2.
\end{align*}
Choosing $\rho_2 = \mL\sqrt{m(m-1)}/\sqrt{2b}$, $\alpha \leq (L+\sqrt{2}\mL m/\sqrt{b})^{-1}$, and $\rho_1 = \delta_s$, we obtain
\begin{align*} \E\|\tilde x^{s+1}-x^\star\|^2 & \leq \left[1 - \tfrac{2\alpha\mu}{1+\alpha(\mu_\phi+\mu)-\rho_1(1+\alpha\mu_\phi)} + \tfrac{\rho_1(1+\alpha\mu_\phi)+c(\alpha)m\delta_s^2}{1+\alpha(\mu_\phi+\mu)-\rho_1(1+\alpha\mu_\phi)} \right] \E\|\tilde x^{s}-x^\star\|^2 \\ & \leq \left[ 1- \tfrac{2\alpha\mu}{1+\alpha(\mu_\phi+\mu)} + \mathcal O(\delta_s) \right]  \E\|\tilde x^{s}-x^\star\|^2 \end{align*}
as $s \to \infty$. This proves q-linear convergence of $\{\tilde x^s\}$ to $x^\star$ in expectation.
\end{proof}
\begin{proof}[Proof of \cref{thm:str_convex_case_obj}]
Let the iteration index $s\in\N_0$ and $k\in\{0,\dots,m-1\}$ be fixed and let $h^{k+1}, \hat x^{k+1}, e_k(x)$ be defined as in the proof of \cref{thm:str_convex_case}. As all $f_i$ are convex we have $M_{\mathcal{S}_k}=M_N=0$. Denote by $\mu_\phi\geq 0$ the strong convexity parameter of $\phi$. Let $x^\star$ denote the unique solution to \cref{prob:deterministic} satisfying $\psi^\star=\psi(x^\star)$. Proceeding as in the proof of \cref{thm:str_convex_case}, we obtain the following analogue to \cref{eqn:strconvex_obj_start}
\begin{align*} 
&(1+\alpha\mu_\phi)(1-\rho_1)\|x^{k+1} - x\|^2 - \|x^{k} - x\|^2 \leq 2\alpha[\psi(x) - \psi(x^{k+1})] + 2\alpha e_k(x) \\ & \hspace{4ex} + \frac{\alpha}{\rho_2} \|\nabla f(x^k)-\nabla f_{\mathcal S_k}(x^k)-v^k\|^2 - \left[1 - (L+\rho_2)\alpha \right] \|x^{k+1}-x^k\|^2 \\ & \hspace{4ex} + {\alpha^2} \|h^{k+1}\|^2 + (1+\alpha\mu_\phi)(\rho_1^{-1}-1) \|\hat x^{k+1}-x^{k+1}\|^2 \end{align*}
for $x\in\mathrm{dom}(\phi)$, $\rho_1\in(0,1)$, and $\rho_2>0$. Conditioned on $x^k$, we have $\E[e_k(x)]=0$. By \cref{prop:inexactness_bound}, it holds $\|h^{k+1}\|^2 \leq (1+\alpha{\mL})^2\tfrac{\bar{A}^2}{\mu_\ast^2b}\epsilon_k^2$ and $\|\hat x^{k+1}-x^{k+1}\|^2 \leq \tfrac{\alpha^2\bar{A}^2}{\mu_\ast^2b}\epsilon_k^2$ and we can again use the estimate $\epsilon_k\leq (2+L)\delta_s\|\tilde x^s -x^\star\|$. Applying \cref{cor:c1} (with $x=x^k, \tilde x = \tilde{x}^s$), conditioned on the history of iterates up to $x^k=x^{s,k}$, we have
\begin{align*}
\E\|\nabla f(x^k)-\nabla f_{\mathcal S_k}(x^k)-v^k\|^2 \leq {4{\mL}}{b^{-1}}(\psi(x^k)-\psi^\star+\psi(\tilde{x}^s)-\psi^\star) \end{align*}
almost surely. At this point, we assume that the condition $1 - (L+\rho_2)\alpha \geq 0$ holds (for the selected $\rho_2$ and $\alpha$). Next, setting $\beta(\rho_1):= (1+\alpha\mu_\phi)(1-\rho_1)$ and $x=x^\star$, we apply expectation
conditioned on the history up to iterate $x^k=x^{s,k}$ and conclude
\begin{align}
\label{eqn:42}
\begin{split}
&\beta(\rho_1)\E\|x^{k+1} - x^\star\|^2 \leq \|x^{k} - x^\star\|^2 + 2\alpha\E[\psi^\star - \psi(x^{k+1})] \\
&\hspace{15ex}+{4\alpha{\mL}}{(\rho_2b)^{-1}}(\psi(x^k)-\psi^\star+\psi(\tilde{x}^s)-\psi^\star)+\tilde{c}(\alpha)\delta_s^2\|\tilde x^s -x^\star\|^2 ,
\end{split}
\end{align}
where $\tilde{c}(\alpha):=  ((1+\alpha{\mL})^2 + (1+\alpha\mu_\phi)\rho_1^{-1})\tfrac{\alpha^2\bar{A}^2}{\mu_\ast^2b} (2+L)^2$.
Using $\|x^k-x^\star\|^2\leq \frac{2}{\mu}(\psi(x^k)-\psi^\star)$ and $1-\beta(\rho_1) = \rho_1+\alpha\mu_\phi(\rho_1-1) \leq\rho_1$, this yields 
\begin{align}
\label{eqn:41}
&\beta(\rho_1)\E\|x^{k+1} - x^\star\|^2 \leq \beta(\rho_1)\|x^{k} - x^\star\|^2 + 2\alpha\E[\psi^\star - \psi(x^{k+1})]\\
&~+\tilde{c}(\alpha)\delta_s^2\|\tilde x^s -x^\star\|^2+\tfrac{4\alpha{\mL}}{\rho_2b}(\psi(x^k)-\psi^\star+\psi(\tilde{x}^s)-\psi^\star) +\tfrac{2\rho_1}{\mu}(\psi(x^k)-\psi^\star). \nonumber
\end{align}
We now require $\rho_1\leq \min\{\frac12,\frac{\alpha\mu{\mL}}{\rho_2 b}\}$.
Using \cref{eqn:41} recursively for $k=1,\dots,m-1$ and \cref{eqn:42} for $k=0$, applying
expectation  conditioned on the history up to $\tilde x^s$ and the tower property, we obtain
\begin{align*}
&\hspace{-11ex}\beta(\rho_1)\E\|x^{m} - x^\star\|^2 +2\alpha(1-\tfrac{3{\mL}}{\rho_2b}){\sum}_{k=1}^{m}\E[\psi(x^k)-\psi^\star] \\
&\hspace{5ex} \leq (1+\tilde{c}(\alpha)m\delta_s^2)\|\tilde x^s -x^\star\|^2 + \tfrac{4\alpha{\mL}(m+1)}{\rho_2b}(\psi(\tilde{x}^s)-\psi^\star).
\end{align*}
Due to the convexity of $\psi$ and by Option II, we can infer $\psi(\tilde{x}^{s+1}) \leq \frac{1}{m}\sum_{k=1}^{m}\psi(x^k)$ and
hence, for $\rho_2>3{\mL}b^{-1}$ it holds that
\begin{align*}
2\alpha(1-\tfrac{3{\mL}}{\rho_2b})m\E[\psi(\tilde{x}^{s+1})-\psi^\star]  \leq (1+\tilde{c}(\alpha)m\delta_s^2)\|\tilde x^s -x^\star\|^2 + \tfrac{4\alpha{\mL}(m+1)}{\rho_2b}(\psi(\tilde{x}^s)-\psi^\star).
\end{align*}
Furthermore, the strong convexity of $\psi$ again implies $\|\tilde x^s -x^\star\|^2\leq \frac{2}{\mu}(\psi(\tilde{x}^s)-\psi^\star)$. We now set $\rho_2=\tfrac{{\mL}}{b}(\tfrac{4}{1-2\theta}+3)$, i.e., $\tfrac{4{\mL}}{\rho_2b}(1-\tfrac{3{\mL}}{\rho_2b})^{-1} = 1-2\theta$. This choice satisfies $\rho_2> 3{\mL}b^{-1}$ automatically 
%
and due to \cref{eqn:strcon_obj_alpha_assumption}, we have $(L+\rho_2)\alpha \leq 1$. This yields
\begin{align*}
\E[\psi(\tilde{x}^{s+1})-\psi^\star]  \leq \left[ \frac{1+\tilde{c}(\alpha)m\delta_s^2}{\mu\alpha(1-\tfrac{3{\mL}}{\rho_2b})m} +\frac{2{\mL}(m+1)}{\rho_2b(1-\tfrac{3{\mL}}{\rho_2b})m}\right](\psi(\tilde{x}^s)-\psi^\star).
\end{align*}
By the choice of $\rho_2$, it holds $\frac{2{\mL}(m+1)}{\rho_2bm}(1-\tfrac{3{\mL}}{\rho_2b})^{-1} \leq 1-2\theta$ for all $m\in\N$. 
Finally, if $\delta_s$ is sufficiently small and $m$ is sufficiently large such that $\tfrac{1+\tilde{c}(\alpha)m\delta_s^2}{\mu\alpha m}(1-\tfrac{3{\mL}}{\rho_2b})^{-1} \leq \theta$, then we can conclude $\E[\psi(\tilde{x}^{s+1})-\psi^\star]  \leq (1-\theta) \E[\psi(\tilde{x}^s)-\psi^\star]$.
\end{proof}
\begin{remark}
In the case $\mu_\phi > 0$, $\rho_1$ can be chosen small enough such that $\beta(\rho_1)\geq 1$ and we obtain the term $1-\tfrac{2{\mL}}{\rho_2b}$ instead of $1-\tfrac{3{\mL}}{\rho_2b}$ in the latter computations.
\end{remark}

\section{Parameter Choices}
\label{sec:table-hyperparams}
Here, we report details on the tuning procedure for \cref{sec:logreg_comparison} and \cref{sec:experiments_student}.
For each method and each dataset, we first identify a candidate interval of step sizes. We then choose a range of step sizes $\alpha$ on this interval (typically $5$--$7$ values) and perform grid search over $2$--$3$ different batch sizes $b$. For \texttt{SAGA}, we additionally try $b=1$.
We select the combination of $\alpha$ and $b$ that performed best in terms of the objective function sub-optimality over three runs. We observe that \texttt{AdaGrad} typically requires larger batch sizes than \texttt{SNSPP}/\texttt{SVRG}, which might be due to a missing variance reduction mechanism in \texttt{AdaGrad}.
\begin{table}[h!]
	\centering
	{\small
	\begin{tabular}{lcccccccc} 
		\cmidrule[1pt](){1-9}
		 \multicolumn{1}{c}{Dataset} 		& \multicolumn{2}{c}{\texttt{SNSPP}} & \multicolumn{2}{c}{\texttt{SAGA}} & \multicolumn{2}{c}{\texttt{SVRG}} & \multicolumn{2}{c}{\texttt{AdaGrad}} \\

		& $\alpha$ & $b$                       & $\alpha$ & $b$                      & $\alpha$ & $b$                      & $\alpha$ & $b$ \\
		\cmidrule(r){1-1} \cmidrule(lr){2-3} \cmidrule(lr){4-5} \cmidrule(lr){6-7} \cmidrule(l){8-9} 
		\texttt{mnist}  & 2.5 & 280 
						& 0.04 & 56 	
						& 0.25 & 280 				 
						& 0.030 & 2800 \\
		
		\texttt{gisette}  & 7.0 & 240 						
						  & 1.20$\cdot 10^{-3}$ & 1
						  & 0.022 & 50
						  & 0.028 & 240 \\
		
		\texttt{sido0}  & 30.0 & 50
					    & 0.20 & 10 
					    & 0.733 & 50 
					    & 0.0150 & 200 \\
		
		\texttt{covtype}  & 50.0 & 50
		    			  & 0.25 & 50 	
		    			  & 0.350 & 50 				
		    			  & 0.10 & 250 \\[0.5ex]
		\hline \\[-1.5ex]
		\texttt{student-t} $(\hat{\nu}=0.5)$  & 1.05 & 20 			
											  & 2.50$\cdot 10^{-3}$ & 1 	
											  & 0.140 & 40 		
											  & 0.032 & 40 \\
		\texttt{student-t} $(\hat{\nu}=1)$  & 3.0 & 20 	
							                & 0.015 & 4
							                & 0.120 & 20 	
							                & 0.030 & 20 \\
		
		\texttt{student-t} $(\hat{\nu}=2)$ & 7.0 & 20 
								           & 0.20 & 20 
								           & 0.40 & 20
								           & 0.032 & 40 \\
		
		\texttt{student-t sido0} & 5.5 & 200 
		 						 & 0.004 & 1 
		 						 & 0.0145 & 10 	
		 						 & 0.015 & 100\\
		\cmidrule[1pt](){1-9}
	\end{tabular}}
	\caption{Step size and batch size values for the experiments in \cref{sec:logreg_comparison} and \cref{sec:experiments_student}.}
	\label{table2:hyperparams}
\end{table}
\section{Additional Plots}
\begin{figure*}[h!]
	\centering
	\begin{subfigure}[t]{0.24\textwidth}
		\centering
		\includegraphics[width=\textwidth]{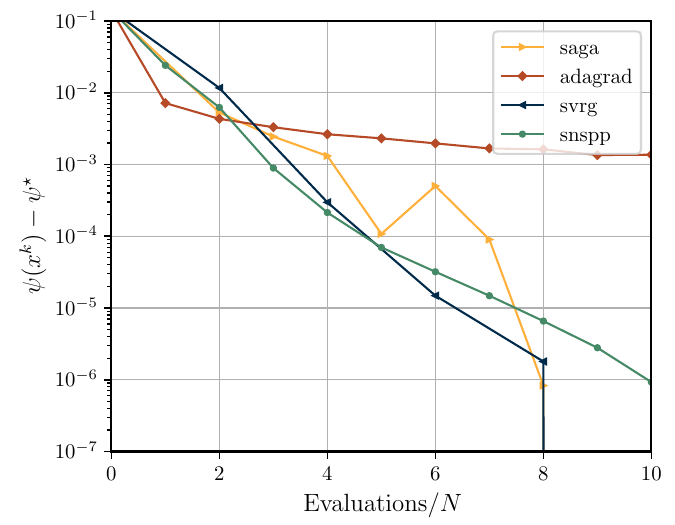}
		\caption{\texttt{mnist}}
		\label{fig:obj2_mnist}
	\end{subfigure}
	\hfill
	\begin{subfigure}[t]{0.24\textwidth}  
		\centering 
		\includegraphics[width=\textwidth]{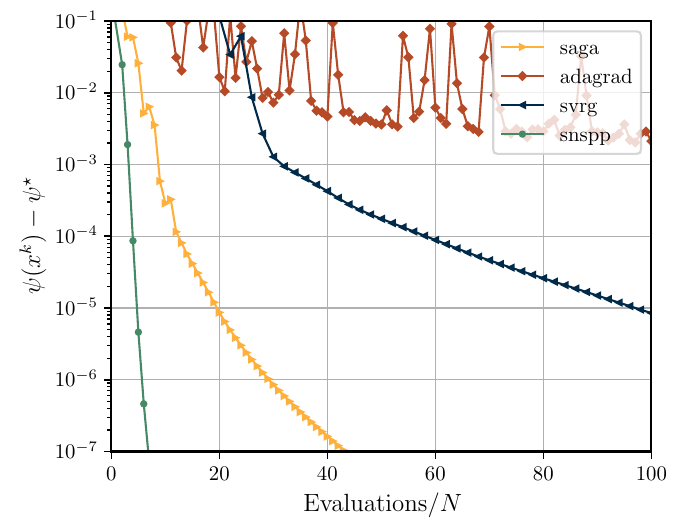}   
		\caption{\texttt{gisette}}
		\label{fig:obj2_gisette}
	\end{subfigure}
	\begin{subfigure}[t]{0.24\textwidth}   
		\centering 
		\includegraphics[width=\textwidth]{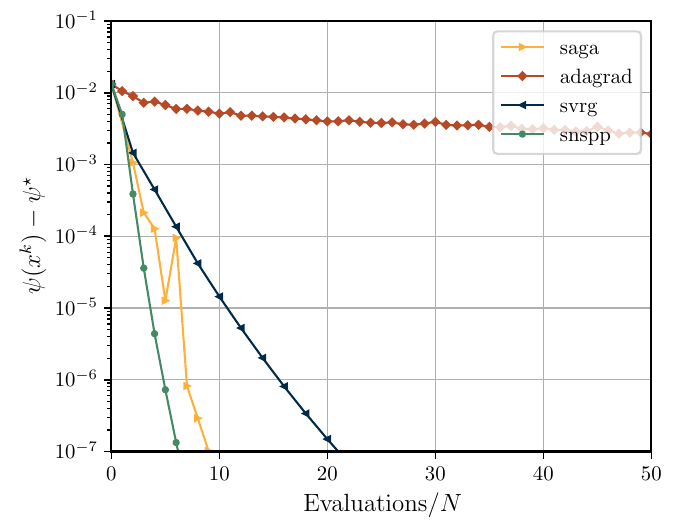}
		\caption{\texttt{sido0}}
		\label{fig:obj2_sido1}
	\end{subfigure}
	\hfill
	\begin{subfigure}[t]{0.24\textwidth}   
		\centering 
		\includegraphics[width=\textwidth]{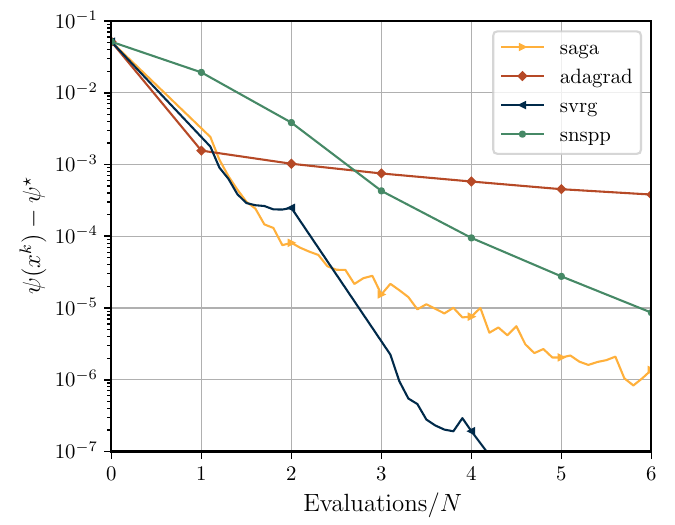}
		\caption{\texttt{covtype}}
		\label{fig:obj2_covtype}
	\end{subfigure}
	\caption{Objective function convergence for the logistic regression datasets with respect to number of gradient evaluations. All settings are identical to \cref{fig:logreg_obj}.
	}
	\label{fig:logreg_obj2}
\end{figure*}
\begin{figure}[t]
	\centering
	\begin{subfigure}[h!]{0.42\textwidth}
		\centering
		\includegraphics[width=\textwidth]{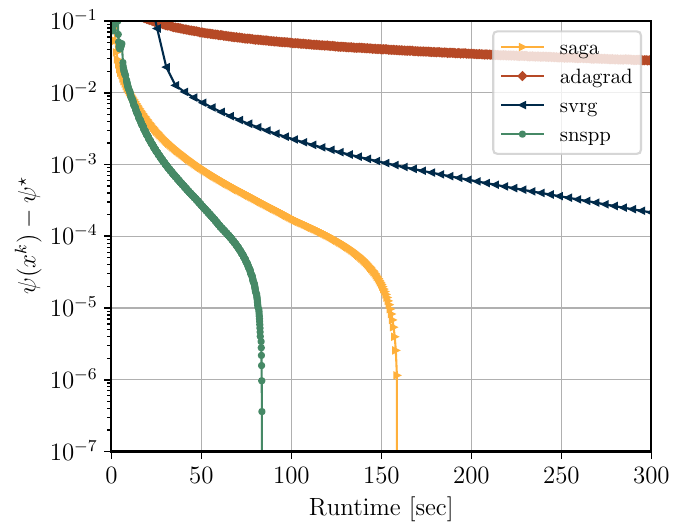}
		\label{fig:madelon_obj}
	\end{subfigure}
	\hfill
	\begin{subfigure}[h!]{0.42\textwidth}  
		\centering 
		\includegraphics[width=\textwidth]{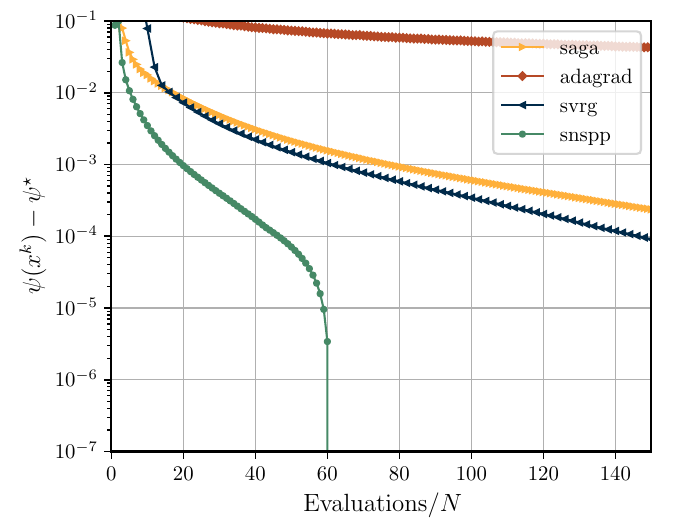}   
		\label{fig:madelon_obj2}
	\end{subfigure}
	\caption{Convergence plot for logistic regression on the \texttt{madelon.2} dataset, with respect to runtime (left) and number of gradient evaluations (right).}
	\label{fig:madelon}
\end{figure}
%

%
%

\bibliographystyle{siamplain}
\setlength{\bibindent}{8pt}
\bibliography{spp_library}
\end{document}

%% file: spp_shared.tex
\usepackage{lipsum}
\usepackage{amsfonts}
\usepackage{amsmath,amssymb}
\usepackage{dsfont}
\usepackage{graphicx}
\usepackage{subcaption} 
\usepackage{epstopdf}
\usepackage[noend]{algorithmic}
\usepackage{float}
\usepackage{enumitem}

\usepackage{comment}
\usepackage{mathtools}

\usepackage{longtable}
\usepackage{booktabs}

\ifpdf
  \DeclareGraphicsExtensions{.eps,.pdf,.png,.jpg}
\else
  \DeclareGraphicsExtensions{.eps}
\fi

\newsiamremark{remark}{Remark}
\newsiamremark{hypothesis}{Hypothesis}
\crefname{hypothesis}{Hypothesis}{Hypotheses}
\newsiamthm{claim}{Claim}
\newsiamremark{assumption}{Assumption}

\headers{A semismooth Newton stochastic proximal point algorithm}{A. Milzarek, F. Schaipp and M. Ulbrich}

\title{A semismooth Newton stochastic proximal point algorithm with variance reduction\thanks{\funding{A. Milzarek is partly supported by the Fundamental Research Fund -- Shenzhen Research Institute of Big Data (SRIBD) Startup Fund JCYJ-AM20190601 and by the Shenzhen Science and Technology Program under Grant GXWD20201231105722002-20200901175001001.}
}}

\author{Andre Milzarek\thanks{School of Data Science (SDS), The Chinese University of Hong Kong, Shenzhen; Shenzhen Research Institute of Big Data (SRIBD) (\email{andremilzarek@cuhk.edu.cn}).}
\and Fabian Schaipp\thanks{Chair of Mathematical Optimization, Department of Mathematics, Technical University of Munich
  (\email{fabian.schaipp@tum.de}, \email{m.ulbrich@tum.de}).}
\and Michael Ulbrich\footnotemark[3]
}

\usepackage{amsopn}

\newcommand{\prox}[1]{\mathrm{prox}_{#1}}

\newcommand{\onehalf}{\frac{1}{2}}
\newcommand{\oneover}[1]{\frac{1}{#1}}
\newcommand{\ixmap}{\kappa}
\newcommand{\trp}[1]{#1^\top}
\newcommand{\mL}{\bar L}

\DeclareMathOperator{\dom}{dom}

\newcommand{\E}{\mathbb{E}}
\newcommand{\R}{\mathbb{R}}

\let\temp\phi
\let\phi\varphi
\let\varphi\temp

\let\epsilon\varepsilon

\definecolor{cuhkb}{RGB}{219,160,1}
\definecolor{cuhkpl}{RGB}{152,24,147} 

\definecolor{bluep}{RGB}{0,128,255}
\definecolor{greenish}{RGB}{24,140,124}
\definecolor{orangy}{RGB}{237,163,26}

\definecolor{darkblue}{RGB}{0, 47, 167}
\hypersetup{
	linkcolor = darkblue,
	citecolor=darkblue,
}

\newcommand{\iprod}[2]{\langle #1, #2 \rangle}
\newcommand{\Rn}{\mathbb{R}^n}
\newcommand{\N}{\mathbb{N}}
\DeclareMathOperator*{\argmin}{argmin}
\DeclareMathOperator*{\argmax}{argmax}
\newcommand{\Spp}{\mathbb{S}^n_{++}}
\newcommand{\Rex}{(-\infty,\infty]}

\newif\ifarxiv
\arxivtrue
